\documentclass[final,12pt]{amsart}

\usepackage{amsmath,amssymb,esint,mathrsfs}
\usepackage{graphicx,xcolor,microtype,needspace,comment}
\usepackage[final]{hyperref}
\allowdisplaybreaks[4]
\hypersetup{hidelinks,
    pdftitle={Strong uniqueness for complete Ricci flows}
}
\theoremstyle{remark}
\numberwithin{equation}{section}
\newtheorem{theorem}{Theorem}[section]

\newtheorem{lemma}[theorem]{Lemma}
\newtheorem{proposition}[theorem]{Proposition}
\newtheorem{corollary}[theorem]{Corollary}

\DeclareMathOperator{\Ric}{Ric}
\DeclareMathOperator{\Rm}{Rm}
\DeclareMathOperator{\diver}{div}
\DeclareMathOperator{\grad}{grad}
\DeclareMathOperator{\supp}{supp}
\DeclareMathOperator{\tr}{tr}
\DeclareMathOperator{\inj}{inj}
\DeclareMathOperator{\Vol}{Vol}
\DeclareMathOperator{\VolB}{VolB}
\DeclareMathOperator{\Area}{Area}

\author[S. Huang]{Shaochuang Huang}
\address[Shaochuang Huang]{School of Science, Shenzhen Campus of Sun Yat-sen
        University, No. 66, Gongchang Road, Guangming District, Shenzhen,
        Guangdong 518107, P. R. China.}
\email{huangshch23@mail.sysu.edu.cn}
\thanks{S. Huang is partially supported by a regular fund from Shenzhen Science
        and Technology Program No. JCYJ20240813151005007 and a start-up
        fund from SYSU}

\author[Z. Peng]{Zhuo Peng}
\address[Zhuo Peng]{School of Science, Shenzhen Campus of Sun Yat-sen
        University, No. 66, Gongchang Road, Guangming District, Shenzhen,
        Guangdong 518107, P. R. China.}
\email{pengzh55@mail2.sysu.edu.cn}

\begin{document}
\title[Pseudolocality and strong uniqueness]{Pseudolocality and strong
        uniqueness for complete Ricci flows}
\subjclass[2020]{Primary 53E20; Secondary 58J35}
\keywords{Ricci flow, pseudolocality, strong uniqueness, entropy, heat kernel}

\begin{abstract}
    In this paper, we prove a uniform pseudolocality theorem for smooth
            complete Ricci flows whose initial metrics have Ricci curvature
            bounded below and uniformly positive injectivity radius.
    No a priori curvature bound is assumed on the flow.
    As a consequence, we obtain strong uniqueness for smooth complete Ricci
            flows from a complete metric with the above geometric bounds.
    We also prove that every smooth complete Ricci flow from a complete
            Ricci-flat metric with uniformly positive injectivity radius is
            stationary.
    The same conclusion holds for any complete flat initial metric.
    In particular, we give an affirmative answer to Chen's question on the
            strong uniqueness of Ricci flow from Euclidean space in every
            dimension \cite{Chen09JDG}.
\end{abstract}

\maketitle

\section{Introduction}

The Ricci flow, introduced by Hamilton \cite{Hamilton82JDG}, evolves
        a Riemannian metric by
\[
    \partial_tg=-2\Ric_g.
\]
A basic problem on a complete noncompact manifold is to control the curvature of
        the flow using geometric information about the initial metric.
Pseudolocality provides such control when the initial geometry is sufficiently
        close to Euclidean geometry at a suitable scale.
In this paper, we prove a uniform pseudolocality theorem for smooth complete
        Ricci flows whose initial metrics have Ricci curvature bounded from below and
        uniformly positive injectivity radius.
No a priori curvature bound is assumed on the evolving metrics, and no almost
        Euclidean isoperimetric or volume condition is assumed  on
        the initial metric.
Strong uniqueness follows as a consequence.
The main analytic ingredient is the monotonicity of Perelman's
        $\mathcal W$-entropy and $\mu$-entropy for smooth complete Ricci flows under
        an initial scalar curvature lower bound.

Perelman's pseudolocality theorem \cite[Section~10]{Perelman02arXiv} shows that on a compact manifold,
        an initial region with a scalar curvature lower bound and an almost
        Euclidean isoperimetric inequality controls curvature of the flow for a
        definite short time.
Chau, Tam and Yu \cite{CTY11CJM} established the complete noncompact version for
        flows with bounded curvature.
Tian and Wang \cite{TW15JAMS} proved a version based on almost nonnegative
        initial Ricci curvature and almost Euclidean volume, while Wang
        \cite{Wang18CambJ,Wang20arXiv} developed local entropy methods for
        noncollapsing and pseudolocality.
Related forms of pseudolocality arise from the 
        $\varepsilon$-regularity theorems of Hein and Naber \cite{HN14CPAM} and
        Bamler \cite{Bamler20arXiv}, which use pointed Nash entropy and pointed
        $\mathcal W$-entropy close to their Euclidean values to obtain local
        curvature control.
Chan, Ma and Zhang \cite{CMZ23JFA} make this connection explicit by combining
        local entropy estimates with Bamler's theorem to derive pseudolocality.

Pseudolocality also plays an important role in constructing Ricci flows on
        noncompact manifolds from initial metrics with possibly
        unbounded curvature.
One constructs suitable approximating flows using Shi's existence theorem
        \cite{Shi89JDG} and then combines uniform local curvature estimates with
        compactness to pass to a limit with the
        prescribed initial metric.
Topping \cite{Topping10JEMS} first used pseudolocality to prove a compactness
        theorem and applied it to existence on surfaces.
Further constructions under various lower curvature bounds and noncollapsing
        assumptions appear, for example, in
        \cite{CRW15JEMS,Hochard16,BCW19Inv,Lai19AdvMath,LT20JDG,ST21GT}.
For constructions using almost Euclidean geometry or small curvature
        concentration, see \cite{Huang19JGA,Wang20arXiv,He24CAG,CHL24AnnPDE,Cheng25Adv,LeeLee} for example. The foregoing list is illustrative rather than exhaustive.

Our concern is curvature control for every smooth complete solution from a
        given initial metric.
The standard complete noncompact pseudolocality theorem assumes bounded
        curvature along the flow.
Some later versions permit curvature to become unbounded as $t\to0^+$ but
        require global curvature bounds on compact positive-time intervals, see
       \cite[Section~2, (2.2)]{Wang20arXiv}  for example.
Results beyond these classes include Chen's local curvature estimates in
        dimensions two and three under suitable initial curvature assumptions
        \cite{Chen09JDG}, Simon's local smoothing theorem in
        dimension two \cite{Simon13GT} and the pseudolocality theorem of Cheng
        and Zhang \cite{CZ25AnnPDE} for $\mathcal L$-complete flows.
$\mathcal L$-completeness of the flow requires existence of minimizing $\mathcal L$-geodesics
        and uniform compact containment of all such minimizers for fixed
        endpoint times and endpoints in a compact set.
Our main theorem below applies in every dimension under the stated initial
        geometric bounds, with neither an a priori global curvature bound on the
        flow nor an additional $\mathcal L$-completeness assumption.
We only require smoothness and completeness of the flow.
Here and below, Ricci flows are smooth on their stated space-time domains,
        locally in space and up to every included time endpoint.
Completeness is assumed for every time slice, including the initial slice.

\begin{theorem}[Uniform pseudolocality]
    \label{thm:main-curvature}
    Given $n\ge2$, $\alpha>0$, $K\ge0$ and $i_0>0$, there is a constant
            $\hat{T}(n,\alpha,K,i_0)>0$ with the following property.
    Let $(M^n,g_0)$ be a smooth complete Riemannian manifold satisfying
    \begin{equation}
        \label{eq:initial-geometry}
        \Ric_{g_0}\ge-K,
        \qquad \inj_{g_0}(x)\ge i_0\quad\text{for every}\quad x\in M.
    \end{equation}
    Then any smooth complete Ricci flow $g(t)$ on $[0,T]$ with
            $g(0)=g_0$ satisfies
    \begin{equation}
        \label{eq:uniform-curvature}
        |\Rm|_{g(t)}\le\frac{\alpha}{t}
        \quad\text{on}\quad M\times
        (0,\min\{\hat{T},T\}].
    \end{equation}
\end{theorem}

The initial conditions in \eqref{eq:initial-geometry} are uniform over $M$, and
        Theorem~\ref{thm:main-curvature} gives a corresponding uniform
        form of pseudolocality.
The time $\hat T$ depends only on $n,\alpha,K,i_0$.
Since $\alpha>0$ is arbitrary, every such flow satisfies
\[
    \lim_{t\to0^+}t\sup_M|\Rm|_{g(t)}=0.
\]
In particular, the initial geometry forces every smooth complete solution into
        the curvature class needed for uniqueness.

Chen and Zhu \cite{CZ06JDG} proved uniqueness for complete Ricci flows with
        uniformly bounded curvature, and Kotschwar \cite{Kotschwar14CAG}
        developed an energy approach to uniqueness and its extensions.
Lee \cite{Lee25arXiv} proved uniqueness for the class of complete solutions with
        curvature bounded by $C/t$ and without smallness restriction on $C$.
Following Chen \cite{Chen09JDG}, we use \emph{strong uniqueness} to mean
        uniqueness among smooth complete solutions without an a priori curvature
        bound or any other assumptions on the evolving metrics.
Chen established strong uniqueness results in dimensions two and three, and
        Topping \cite{Topping15GT} proved strong uniqueness for instantaneously
        complete flows on arbitrary surfaces.
Cheng and Zhang \cite{CZ25AnnPDE} proved stationarity from the Euclidean
        metric under $\mathcal L$-completeness.
Related strong uniqueness results for a class of instantaneously complete
        Chern-Ricci flows were obtained by the authors in \cite{HP26arXiv}.
By applying Theorem~\ref{thm:main-curvature} together with Lee's uniqueness theorem
        \cite{Lee25arXiv}, we obtain the following strong uniqueness for smooth
        complete Ricci flows.

\begin{corollary}[Strong uniqueness]
    \label{cor:uniqueness}
    Let $(M^n,g_0)$ be complete and satisfy \eqref{eq:initial-geometry} for
            $n\ge2$, $K\ge0$ and $i_0>0$.
    If $g(t)$ and $\widetilde g(t)$ are smooth complete Ricci flows on $[0,T]$
            with $g(0)=\widetilde g(0)=g_0$, then
    \[
        g(t)=\widetilde g(t)
        \qquad\text{for }0\le t\le
            \min\{\hat{T}(n,1,K,i_0),T\}.
    \]
\end{corollary}

As a consequence, we have the following stationarity result for smooth
        complete Ricci flows.

\begin{corollary}
    \label{cor:ricci-flat}
    Let $(M,g_0)$ be complete and Ricci-flat, and suppose that
            $\inf\limits_M\inj_{g_0}>0$.
    Every smooth complete Ricci flow $g(t)$ on $[0,T]$ with $g(0)=g_0$ satisfies
            $g(t)=g_0$ for every $t\in[0,T]$.
\end{corollary}

For a complete flat initial metric, lifting the flow to its Euclidean universal
        cover removes the requirement for an injectivity radius hypothesis on
        the original manifold.

\begin{corollary}
    \label{cor flat uniqueness}
    Every smooth complete Ricci flow from a complete flat metric is stationary
            throughout its existence time interval.
    In particular, every such flow from the Euclidean metric on
            $\mathbb R^n$ is stationary.
\end{corollary}

In particular, we give an affirmative answer to Chen's question 
        on the strong uniqueness of Ricci flow from Euclidean space in every dimension \cite{Chen09JDG}. The case of dimensions two and three has been solved by Chen in the same work \cite{Chen09JDG}.

\medskip
\noindent\textbf{Strategy of the proof.}
The proof uses Perelman's entropy \cite{Perelman02arXiv} to obtain curvature
        control from the initial geometry.
This approach is closely related to Wang's entropy formulation of
        pseudolocality \cite{Wang20arXiv}, see also Lee, Naber and Neumayer \cite{LNN23GT} and Cheng \cite{Cheng25Adv}.
We first explain how Lemma~\ref{lem uniform small entropy} and
        Corollary~\ref{cor initial entropy propagation} imply
        Theorem~\ref{thm:main-curvature}.
In Theorem~\ref{thm:curvature}, we prove the curvature estimate assuming an
        initial scalar curvature lower bound and
\[
    \mu(g_0,s)\ge-\omega(s),\qquad
    \omega(s)\to0\quad\text{as}\quad s\to0^+.
\]
Under \eqref{eq:initial-geometry}, Lemma~\ref{lem uniform small entropy}
        gives this entropy bound, with $\omega$ depending only on $n,K,i_0$. We would like to point out that the geometric assumptions in \eqref{eq:initial-geometry} are used only in Lemma ~\ref{lem uniform small entropy}  apart from providing the initial scalar curvature lower bound.
To prove the lemma, we use the harmonic radius estimate of Anderson and
        Cheeger \cite{AC92JDG} to obtain uniform coordinate charts at
        sufficiently small scales.
Euclidean isoperimetry and Wang's comparison
        \cite[Lemma~3.5]{Wang18CambJ}
        then give local entropy lower bounds.
Using a partition of unity with $\sum\limits_j\eta_j^2=1$, we obtain the
        required global estimate.
Since $\mathcal R_{g_0}\ge-nK$, it remains to prove
        Theorem~\ref{thm:curvature}.

For this, we use the entropy monotonicity in
        Corollary~\ref{cor initial entropy propagation}:
\[
    \mu(g(t),\tau)\ge\mu(g_0,t+\tau),\qquad \tau>0.
\]
Suppose the curvature estimate in Theorem~\ref{thm:curvature} is false.
We choose points of large curvature at times tending to zero and rescale the
        corresponding flows.
A spatially weighted point-picking argument, as in
        \cite[Section~4]{Wang20arXiv}, together with the shrinking balls lemma
        \cite[Corollary~3.3]{ST22JDG}, gives uniform curvature bounds on
        parabolic neighborhoods whose spatial radii tend to infinity after
        rescaling.
Corollary~\ref{cor initial entropy propagation} propagates the initial entropy bound to the rescaled flows. Together with the local curvature bounds, this gives uniform volume and injectivity radius lower bounds at the basepoints. Hamilton’s compactness theorem then yields a complete pointed limit flow $(M_\infty,g_\infty(t),p_\infty)$ with
\[
|\mathrm{Rm}|_{g_\infty(0)}(p_\infty)=1.
\]
Moreover, after passing to the limit,  the initial entropy bound together with Corollary~\ref{cor initial entropy propagation} gives
\[
\mu(g_\infty(0),\tau)\ge0
\qquad\text{for every }\tau>0.
\]
Cheng's entropy rigidity theorem \cite{Cheng24arXiv}, in the form of
        Lemma~\ref{lem:entropy-rigidity}, implies that $g_\infty(0)$ is flat.
This contradicts \[
|\mathrm{Rm}|_{g_\infty(0)}(p_\infty)=1,
\]
which proves
        Theorem~\ref{thm:curvature}, and hence Theorem~\ref{thm:main-curvature}.

The main analytic step is to establish the monotonicity for $\mathcal W$-entropy needed
        for Corollary~\ref{cor initial entropy propagation} under only an
        initial scalar curvature lower bound.
On closed manifolds, Perelman's first variation formulas give the
        monotonicity for $\mathcal F$ and $\mathcal W$
        \cite[Sections~1 and~3]{Perelman02arXiv}.
For complete noncompact flows with bounded curvature, $\mathcal W$-entropy monotonicity under suitable decay assumptions is discussed in \cite[Section~3.3]{LNN23GT}. Related noncompact results appear in \cite[Section~7]{CTY11CJM} for conjugate heat kernels and in \cite[Section~4]{Zhang12JFA} for terminal data arising from $\mathcal W$-entropy extremals. Wang \cite{Wang18CambJ,Wang20arXiv} develops almost monotonicity estimates
        for local entropy by evolving local minimizers and controlling the
        localization errors, this also requires care with terminal regularity.
In this work, we will prove monotonicity for global  $\mathcal W$-entropy and $\mu$-entropy without an a priori
        global curvature bound on the flow.
Local smoothness of integrand alone does not justify differentiation under an integral
        over a noncompact manifold $M$, and the energy and entropy integrals are not known a priori
        to be finite or continuous in time. Our main contribution is to 
establish their finiteness and the time regularity needed for monotonicity
        through integral identities and inequalities under a smooth complete Ricci flow  from a complete initial metric with scalar curvature lower bound.

Assume that $\mathcal R_{g(0)}\ge-K_*$ for some $K_*\ge0$.
Let $u$ be the minimal nonnegative solution of the conjugate heat equation with
        terminal value $\phi^2$, where $\phi\in C_c^\infty(M)$ is nonnegative
        and $\int_M\phi^2\,d\mu_{g(T)}=1$.
Write $[v]_t=\int_Mv(x,t)u(x,t)\,d\mu_{g(t)}(x)$, and consider
\[
    I(t)=[|\nabla\log u|^2]_t,\qquad
    F(t)=I(t)+[\mathcal R]_t,\qquad
    H(t)=-[\log u]_t.
\]
The terminal values are defined using $\phi$.
For $\theta(t)=\tau+T-t$, $\tau>0$, the $\mathcal W$-entropy decomposes as
\[
    \mathcal W(t):=\mathcal W(g(t),\sqrt{u(t)},\theta(t))
    =\theta(t)F(t)+H(t)-\frac n2\log(4\pi\theta(t))-n.
\]
Thus our goal is to prove  monotonicity for $F$ and an integral identity for $H$ that together imply monotonicity for $\mathcal W$.

Dirichlet exhaustion, the maximum principle and local parabolic estimates
        give conservation of mass and local smooth convergence
        $u(t)\to\phi^2$ as $t\to T^-$.
Neither conclusion uses the initial scalar curvature lower bound.
Using the initial scalar lower bound and Chen's estimate \cite{Chen09JDG},
        we then control $[\mathcal R]_t$ and the spacetime integral of
        $|\Ric|^2$ against $u\,d\mu_{g(t)}\,dt$.
The scalar curvature average is nondecreasing and of bounded variation.

For the energy term $F(t)$, we use the variational representation of the Fisher
        information $I(t)$: for $t<T$,
\[
    I(t)=\sup_{X\in C_c^\infty(TM)}[2\diver X-|X|^2]_t.
\]
Starting from a compactly supported vector field, we solve
        $\partial_tX=\Delta_{g(t)}X+\Ric_{g(t)}(X)$,
        an equation used by Ilmanen and Knopf \cite{IK03MRL} and Brendle
        \cite{Brendle20Acta}.
The identity $\Box|X|_{g(t)}^2=-2|\nabla X|_{g(t)}^2$ gives a uniform bound
        for $X$.
Together with conservation of mass, this bound allows us to pass to the
        limit in the weighted norm identity, since the boundary terms tend
        to zero under exhaustion.
A divergence estimate and the scalar curvature integral inequality also give an
        exact integral identity for $[\diver X]_t$.
Combining these identities with the scalar curvature inequality yields
        monotonicity of $J_X(t):=[\mathcal R+2\diver X-|X|^2]_t$.
The compact support of the terminal data bounds $J_X$ at $T$.
Taking the supremum over the initial vector fields first proves finiteness
        of $I$ and $F$, and then shows that $F$ is bounded, nondecreasing,
        and satisfies
\[
    dF\ge\frac2nF^2\,dt
\]
as an inequality of measures.
In particular, both $F$ and $I=F-[\mathcal R]$ have bounded variation.
We also prove that $I(t)$ and $F(t)$ converge to their terminal values as
        $t\to T^-$.
We do not require absolute continuity of $F$, the measure $dF$ may have a
        nonnegative singular part.

To our knowledge, this use of the variational representation of Fisher information, together with evolving test vector fields, to establish the finiteness of $I$ and $F$ and the monotonicity of $F$ is new in the Ricci flow setting. By deriving integral estimates before taking the supremum over test vector fields, we obtain these conclusions under an initial scalar curvature lower bound, without assuming global curvature bounds or a priori finiteness and time regularity of these functionals.

For the entropy term $H(t)$, the identity $H'=-F$ is classical on closed flows,
        see \cite[equation~(6.5.1)]{Topping06Book} and
        \cite[Theorem~2.1]{GPT13PJM}.
To justify it here, we first prove an integration by parts formula for bounded
        smooth functions, using Dirichlet heat solutions as cutoffs.
By applying this  formula to bounded regularizations of $\log(U/u)$, where $U$
        is an upper bound for $u$, we prove that $H$ is finite and absolutely
        continuous, in fact Lipschitz, and
\[
    H(t)-H(s)=-\int_s^tF(r)\,dr.
\]
The measure inequality for $F$ and this identity give
\[
    d\mathcal W\ge\frac{2\theta}{n}
        \left(F-\frac{n}{2\theta}\right)^2dt\ge0.
\]
Together with the endpoint estimates, this proves the required
        $\mathcal W$-monotonicity.

To pass from $\mathcal W$ to $\mu$, we must account for the noncompact support
        of $\sqrt{u(0)}$.
The established integrability of $|\nabla\sqrt{u(0)}|^2$,
        $|\mathcal R_{g(0)}|u(0)$, and $u(0)|\log u(0)|$ permits approximation
        by normalized compactly supported smooth functions with convergence
        of every term in $\mathcal W$.
Taking the infimum over the terminal data then gives
        Corollary~\ref{cor initial entropy propagation}.

\medskip

Throughout the paper, manifolds and Riemannian metrics are smooth, and manifolds
        are connected and without boundary.
The symbols $\Vol_g$ and $\Area_g$ denote Riemannian volume and hypersurface
        area, and $g_{\mathrm E}$ denotes the Euclidean metric.
We set $\VolB_g(p,r):=\Vol_g(B_g(p,r))$.
We write $\nabla$, $\Ric$, and $\Rm$ for the Levi-Civita connection, Ricci
        tensor, and full curvature tensor, respectively.
The scalar curvature of a metric $g$ is denoted by $\mathcal R_g$. In
        particular, $\mathcal R_{g(t)}$ is the scalar curvature of $g(t)$, and
        $\mathcal R_{g(t)}(x)$ is its value at $x$.
The scalar and rough Laplacians use the convention $\Delta_g:=\tr_g\nabla^2$.
For a time-dependent metric $g(t)$, set
\[
    \Box:=\partial_t-\Delta_{g(t)},
\]
where $\Delta_{g(t)}$ is the Laplacian with respect to $g(t)$.
We denote the formal adjoint of $\Box$ with respect to
        $d\mu_{g(t)}\,dt$ by $\Box^*$.
Within computations, we may omit metric subscripts and metric or time arguments
        when they are unambiguous, writing, for example,
        $\mathcal R$ and $\Delta$.
Spatial operators, norms, inner products, contractions, and volume measures at
        time $t$ are taken with respect to $g(t)$ unless otherwise stated.
When a tensor or differential operator is defined using the same metric as its
        norm, we display the metric only on the norm.
In particular, we use the abbreviations
\[
    |\Rm|_g:=|\Rm_g|_g,
    \qquad |\Ric|_g:=|\Ric_g|_g.
\]
The same convention applies to covariant derivatives, inner products, and traces
        when the underlying metric is clear.
For weighted averages, the subscript specifies the time and the metric, thus we
        write $[\mathcal R]_t$ and $[|\nabla X|^2]_t$.
These conventions apply also in definitions and theorem statements.
When the defining metric and the norm metric differ, both are retained, we write $|\Rm_g|_{\widehat g}$ for
        example.
We use $\Ric_g$ also for the endomorphism of $TM$ obtained by
        raising one index, so that
\[
    (\Ric_g(X))^i:=g^{ij}R_{jk}X^k.
\]
Here $R_{jk}$ are the covariant components of the Ricci tensor.
Thus $g(\Ric(X),Y)=\Ric_g(X,Y)$ for vector fields $X,Y$.
When adding $\nabla X$ to a covariant two-tensor, we lower the vector index, so
        $(\nabla X)_{ij}:=g_{jk}\nabla_iX^k$.
This two-tensor need not be symmetric.
We write $\langle\cdot,\cdot\rangle_g$ for the metric inner product
        on vectors and tensors.
For covariant two-tensors $A,B$,
\[
    \langle A,B\rangle_g=g^{ik}g^{j\ell}A_{ij}B_{k\ell}.
\]

The paper is organized as follows. In Section~\ref{sec:preliminaries}, we study the heat and conjugate heat equations and establish the weighted integral identities and inequalities needed in the proof. In Section~\ref{sec:entropy}, we establish the finiteness and time regularity of the energy and entropy and prove the monotonicity of $\mathcal W$ and $\mu$. In Section~\ref{sec:curvature}, we establish a uniform entropy lower bound for the initial metric and prove the pseudolocality theorem. Finally, in Section~\ref{sec:consequences}, we prove the strong uniqueness and stationarity results.

\medskip

{\bf AI disclosure statement.} 
This work grew out of the authors' sustained study of complete noncompact
        Ricci flows.
The authors used ChatGPT-6 Astra in an extended series of mathematical
        discussions, posing research questions, supplying relevant literature,
        and suggesting directions of investigation.
In the course of these discussions, the model contributed the main ideas
        underlying the proofs.
The authors critically examined and verified the arguments,
        checked the use of results from the literature,
        and reorganized and refined the proofs.
They also developed the final formulation and organization of the results,
        and wrote the manuscript.
The authors take full responsibility for the mathematical content,
        the accuracy of the citations, and the final presentation.

\section{Heat equations, integral inequalities and identities}
\label{sec:preliminaries}

Throughout this section, let $(M^n,g(t))$ be a complete
        Ricci flow for $t\in[0,T]$.
We first recall some maximum principles and heat kernel facts, and
        then derive the weighted integral   inequalities and identities needed for the analysis of 
        Perelman's energy and entropy.

\subsection{Maximum principles and heat kernels}

We first recall Chen's scalar curvature lower bound, which requires
        no a priori bound on $|\Rm|_{g(t)}$.

\begin{proposition}[Chen {\cite[Corollary~2.3~(i)]{Chen09JDG}}]
    \label{prop scalar lower bound}
    If $\mathcal R_{g(0)}\ge-K_*$ for some $K_*\ge0$, then
    \begin{equation}
        \label{eq:chen-scalar}
        \mathcal R_{g(t)}\ge
            -\frac{K_*}{1+2K_*t/n}\ge-K_*
        \qquad\text{on }M\times[0,T].
    \end{equation}
\end{proposition}

The next proposition is a special case of
        \cite[Proposition~2.1]{HP26arXiv} by the authors.
Their proof adapts the localization argument of Chen
        \cite[Proposition~2.1]{Chen09JDG}.
See also the related maximum principles in
        \cite[Lemma~5.1]{HLTT22CAG} by the first-named author, Lee, Tam
        and Tong, and in \cite{CL23MathAnn} by Cho and Li.

\begin{proposition}
    \label{prop quadratic maximum}
    Let $0\le s<T$, $\beta>0$, and $K\ge0$.
    Suppose that $w\in C(M\times[s,T])$ is smooth for $t>s$ and satisfies
    \[
        \Box w\le-\beta w^2\quad\text{wherever}\quad w>0,
        \qquad w(\cdot,s)\le K.
    \]
    Then
    \[
        w(x,t)\le\frac{K}{1+\beta K(t-s)}\le K
        \quad\text{for all}\quad(x,t)\in M\times[s,T].
    \]
\end{proposition}

\begin{proof}
    By applying \cite[Proposition~2.1]{HP26arXiv} to $-w$,
            taking $E\equiv0$, $W\equiv0$, and $\delta=\beta$
            in notation there, the result follows.
\end{proof}

The formal adjoint of $\Box$ with respect to $d\mu_{g(t)}\,dt$ is
\[
    \Box^*:=-\partial_t-\Delta_{g(t)}+\mathcal R_{g(t)}.
\]
Let $K(x,t;y,s)$ be the minimal positive fundamental solution of
        $\Box$ with $s<t$.
Its existence and monotone Dirichlet exhaustion construction follow
        from \cite[Theorem~24.40 and
        equations~(24.106)--(24.107)]{CGIKLN10Book} by taking $Q=0$
        in their notation.
        If $M$ is compact, these are standard heat kernel facts.
The kernel satisfies
\[
    \partial_tK=\Delta_{g(t),x}K,
    \qquad
    \partial_sK=-\Delta_{g(s),y}K+\mathcal R_{g(s)}(y)K.
\]
A \emph{smooth exhaustion} $\{\Omega_j\}$ of $M$ consists of smooth
        domains $\Omega_j\subset\subset\Omega_{j+1}\subset\subset M$
        with $\bigcup\limits_j\Omega_j=M$.
By the Dirichlet maximum principle and monotone convergence, we have
\[
    0\le\int_MK(x,t;y,s)\,d\mu_{g(s)}(y)\le1
    \quad\text{for}\quad 0\le s<t\le T.
\]

The following conservation result is a consequence of
        Kuwada-Philipowski \cite[Theorem~1]{KP11MathZ}.
For the reader's convenience, we include a direct proof.
For related conservation results under additional global curvature and
        derivative bounds along the flow, see \cite{CTY11CJM, HT22JGA}.

\begin{lemma}
    For all $x\in M$ and $0\le s<t\le T$,
    \begin{equation}
        \label{eq heat conservation}
        \int_M K(x,t;y,s)\,d\mu_{g(s)}(y)=1.
    \end{equation}
\end{lemma}

\begin{proof}
    Fix $s<T$ and set
    \[
        w(x,t):=e^{-(t-s)}\left(1-\int_MK(x,t;y,s)\,d\mu_{g(s)}(y)\right)
        \quad\text{for}\quad s<t\le T.
    \]
    By Dirichlet exhaustion and local parabolic regularity, the heat
            kernel integral is smooth and solves the heat equation for $t>s$.
    It tends to $1$ locally uniformly as $t\to s^+$, so $w$ extends
            continuously to $w(\cdot,s)=0$.
    Moreover,
    \[
        0\le w\le1,\qquad \Box w=-w\le-w^2.
    \]
    Proposition~\ref{prop quadratic maximum}, with $\beta=1$ and
            $K=0$, yields $w\le0$.
    Hence $w\equiv0$, which completes the proof.
\end{proof}

\subsection{The conjugate heat equation and integral inequalities}

Fix a nonnegative $\phi\in C_c^\infty(M)$ with
        $\int_M\phi^2\,d\mu_{g(T)}=1$, and define
\begin{equation}
    \label{eq conjugate definition}
    u(y,s):=\int_M K(x,T;y,s)\phi(x)^2\,d\mu_{g(T)}(x),
    \qquad 0\le s<T.
\end{equation}
Set $u(T):=\phi^2$. For $s<T$, the function $u$ is positive and smooth. 
It is the minimal nonnegative solution of the conjugate heat equation
        satisfying
\begin{equation}
    \label{eq conjugate equation}
    \partial_tu=-\Delta_{g(t)}u+\mathcal R_{g(t)}u,\qquad
    \int_Mu(t)\,d\mu_{g(t)}=1.
\end{equation}
Indeed, choose a smooth exhaustion $\{\Omega_j\}$ of $M$ with
        $\supp\phi\subset\subset\Omega_1$.
Let $K_j$ be the Dirichlet heat kernel of $\Box$ on $\Omega_j$,
        extended by zero in both spatial variables outside $\Omega_j$.
If $M$ is compact, take $\Omega_j=M$ and $K_j=K$.
Define
\[
    u_j(y,s):=\int_{\Omega_j}K_j(x,T;y,s)\phi(x)^2
            \,d\mu_{g(T)}(x).
\]
Then $u_j$ solves the conjugate heat equation with terminal value
        $\phi^2$ and zero boundary values, see
        \cite[Lemmas~26.4 and~26.9]{CGIKLN10Book}.
Since $0\le K_j\to K$, it follows by monotone convergence that
\[
    0\le u_j(y,s)\to u(y,s).
\]
If $\widetilde u$ is another nonnegative solution with terminal value $\phi^2$,
        then $u_j\le\widetilde u$ by the parabolic maximum
        principle on $\Omega_j$.
Letting $j\to\infty$, we conclude that $u\le\widetilde u$.
Finally, by Tonelli's theorem and \eqref{eq heat conservation}, we have
\begin{align*}
    \int_Mu(y,s)\,d\mu_{g(s)}(y)
    &=\int_M\phi(x)^2
        \left(\int_MK(x,T;y,s)\,d\mu_{g(s)}(y)\right)
        d\mu_{g(T)}(x)\\
    &=\int_M\phi(x)^2\,d\mu_{g(T)}(x)=1.
\end{align*}

We also establish local regularity up to the terminal time.
By \cite[Corollary~2.3~(i)]{Chen09JDG}, we have
        $\mathcal R_{g(t)}\ge-n/(2t)\ge-n/T$ on $M\times[T/2,T]$.
By parabolic maximum principle after time reversal, we have
\[
    0\le u_j(y,t)\le e^{n(T-t)/T}\sup_M\phi^2
        \le e^{n/2}\sup_M\phi^2
    \quad\text{for all}\quad T/2\le t\le T.
\]
Fix $k$.
On the fixed domain $\Omega_{k+1}\times[T/2,T]$, the equations for
        $u_j$, $j\ge k+1$, have smooth coefficients independent of $j$.
After time reversal, they are uniformly parabolic with common
        initial data $\phi^2$.
The uniform bound above and local parabolic Schauder estimates (see,
        for example, \cite[Theorem~4.28]{Lieberman96Book})
        up to the initial time yield uniform estimates of every order
        for $u_j$ on $\Omega_k\times[T/2,T]$, independent of $j$.
We pass to the limit and conclude that $u$ extends locally smoothly
        to $t=T$, with terminal value $\phi^2$.
In particular,
\[
    u(t)\to\phi^2=u(T)
    \quad\text{in }C^\infty_{\mathrm{loc}}(M)
    \quad\text{as}\quad t\to T^-.
\]

Moreover, if $\mathcal R_{g(t)}\ge-K_*$ on $M\times[0,T]$ for some
        $K_*\ge0$, by the same argument and letting $j\to\infty$,
        we obtain
\begin{equation}
    \label{eq conjugate upper}
    0<u(t)\le e^{K_*(T-t)}\sup\limits_M\phi^2
        \le e^{K_*T}\sup\limits_M\phi^2=:U
        \quad\text{for all}\quad 0\le t<T.
\end{equation}

A \emph{probability density} with respect to $d\mu_g$ is a measurable
        function $v\ge0$ satisfying $\int_Mv\,d\mu_g=1$.
By \eqref{eq conjugate equation}, $u(t)$ is a probability density
        with respect to $d\mu_{g(t)}$.
Thus
\[
    d\nu_t(x):=u(x,t)\,d\mu_{g(t)}(x)
\]
defines a probability measure on $M$.
For related uses of conjugate heat equations and their associated probability measures in Ricci
        flow, see \cite{Perelman02arXiv,HN14CPAM,Bamler20arXiv} for example.
For a measurable function $v(x,t)$ with $\int_M|v(x,t)|\,d\nu_t(x)<+\infty$,
        define its weighted average by
\[
    [v]_{u,t}
    :=\int_Mv(x,t)\,d\nu_t(x)
    =\int_Mv(x,t)u(x,t)\,d\mu_{g(t)}(x).
\]
For $v\ge0$, the same notation allows the value $+\infty$.
When $u$ is fixed, write $[v]_t:=[v]_{u,t}$.

In general, smoothness of $u$ and $v$ and finiteness of these integrals at each
        time do not, on a noncompact manifold, imply continuity or
        differentiability of $t\mapsto[v]_t$.
For a smooth $v$ with
        $\bigcup\limits_{t\in[0,T]}\supp v(t)\subset\subset M$ and $0<t<T$,
\begin{align*}
    \frac{d}{dt}[v]_t
    &=\int_M\bigl(u\partial_tv+v\partial_tu
                    -\mathcal Ruv\bigr)\,d\mu_{g(t)}\\
    &=\int_M\bigl(u\partial_tv-v\Delta u\bigr)\,d\mu_{g(t)}\\
    &=\int_Mu(\partial_tv-\Delta v)\,d\mu_{g(t)}
    =[\Box v]_t.
\end{align*}
For functions without compact support, the next lemma proves an
        integral inequality.
No scalar curvature lower bound is assumed,
        $u$ is defined by \eqref{eq conjugate definition}.

\begin{lemma}
    \label{lem positive comparison}
    Let $v\ge0$ be smooth on $M\times[0,T]$ and satisfy $\Box v=q\ge0$.
    For $0\le t_1<t_2\le T$, it holds that
    \[
        [v]_{t_2}\ge[v]_{t_1}
            +\int_{t_1}^{t_2}[q]_t\,dt.
    \]
\end{lemma}

\begin{proof}
    Consider the exhaustion $\{\Omega_j\}$ and the Dirichlet solutions
            $u_j$ constructed above.
    Note that $\supp\phi\subset\subset\Omega_j$ for each $j$.
    Each $u_j$ is smooth on $\overline{\Omega}_j\times[0,T]$.
    Let $V$ be the outward unit normal with respect to $g(t)$ and let
            $d\sigma_{g(t)}$ be the induced boundary measure.
    Since $u_j\ge0$ and $u_j=0$ on $\partial\Omega_j$, we have
            $\partial_Vu_j\le0$.
    By the conjugate heat equation and Green's formula, we obtain
    \begin{align*}
        \frac{d}{dt}\int_{\Omega_j}vu_j\,d\mu_{g(t)}
        &=\int_{\Omega_j}(u_j\partial_tv-v\Delta u_j)\,d\mu_{g(t)}\\
        &=\int_{\Omega_j}qu_j\,d\mu_{g(t)}
            -\int_{\partial\Omega_j}v\partial_Vu_j\,d\sigma_{g(t)}\\
        &\ge\int_{\Omega_j}qu_j\,d\mu_{g(t)},
    \end{align*}
    where the last inequality follows from $v\ge0$.
    Hence, for $0\le t_1<t_2\le T$,
    \[
        \int_{\Omega_j}v(t_2)u_j(t_2)\,d\mu_{g(t_2)}
        \ge\int_{\Omega_j}v(t_1)u_j(t_1)\,d\mu_{g(t_1)}
            +\int_{t_1}^{t_2}\int_{\Omega_j}qu_j\,d\mu_{g(t)}\,dt.
    \]
    The zero extensions of $u_j$ increase to $u$ on $M\times[0,T]$.
    Since $v,q\ge0$, the result follows by monotone convergence as
            $j\to\infty$.
\end{proof}

\begin{proposition}
    \label{prop scalar comparison}
    Assume that $\mathcal R_{g(t)}\ge-K_*$ on $M\times[0,T]$ for
            some $K_*\ge0$.
    Then $[|\mathcal R|]_t<+\infty$ for every $t\in[0,T]$,
            and, for $0\le t_1<t_2\le T$,
    \begin{equation}
        \label{eq scalar energy}
        [\mathcal R]_{t_2}-[\mathcal R]_{t_1}
        \ge2\int_{t_1}^{t_2}[|\Ric|^2]_t\,dt.
    \end{equation}
    In particular,
    \begin{equation}
        \label{eq ricci integrable}
        \int_0^T[|\Ric|^2]_t\,dt
        \le\frac12\bigl([\mathcal R]_T+K_*\bigr)<+\infty.
    \end{equation}
    The function $t\mapsto[\mathcal R]_t$ is
            nondecreasing on $[0,T]$ and left-continuous on $(0,T]$, with
    \begin{equation}
        \label{bdd ricci integrable}
        -K_*\le[\mathcal R]_t
            \le[\mathcal R]_T.
    \end{equation}
\end{proposition}

\begin{proof}
    Since
    \[
        \mathcal R+K_*\ge0
        \quad\text{and}\quad\Box(\mathcal R+K_*)=2|\Ric|^2\ge0,
    \]
    by applying Lemma~\ref{lem positive comparison}  to $\mathcal R+K_*$,  we have for every $0\le t<T$,
    \begin{align*}
        0\le{}&[\mathcal R+K_*]_t
            +2\int_t^T[|\Ric|^2]_\tau\,d\tau\\
        \le{}&\int_M(\mathcal R_{g(T)}+K_*)\phi^2
                    \,d\mu_{g(T)}<+\infty.
    \end{align*}
    The last integral is finite because $\phi$ has compact support.
    Since $|\mathcal R|\le\mathcal R+2K_*$ and $[1]_t=1$,
    \[
        [|\mathcal R|]_t
        \le[\mathcal R+K_*]_t+K_*<+\infty.
    \]
    Thus $\mathcal R$ is integrable with respect to $d\nu_t$ and
            $[\mathcal R+K_*]_t=[\mathcal R]_t+K_*$.
    We apply Lemma~\ref{lem positive comparison} again on $[t_1,t_2]$
            and cancel $K_*$ to obtain \eqref{eq scalar energy}.
    Monotonicity and \eqref{bdd ricci integrable} follow.
    By taking $t_1=0$, $t_2=T$, we obtain \eqref{eq ricci integrable}.

    The local smoothness of $u$ up to $T$ allows us to apply Fatou's lemma
            for any $t_k\to t\in[0,T]$:
    \begin{align*}
        [\mathcal R]_t+K_*
        &=\int_M\lim_{k\to\infty}
            \left((\mathcal R(x,t_k)+K_*)u(x,t_k)
                \frac{d\mu_{g(t_k)}}{d\mu_{g(t)}}(x)\right)
                d\mu_{g(t)}(x)\\
        &\le\liminf_{k\to\infty}
            \bigl([\mathcal R]_{t_k}+K_*\bigr).
    \end{align*}
    By this lower semicontinuity and monotonicity,  we have for $0<t\le T$,
    \[
        [\mathcal R]_t
        \le\liminf_{s\to t^-}[\mathcal R]_s
        \le\limsup_{s\to t^-}[\mathcal R]_s
        \le[\mathcal R]_t.
    \]
    This proves left continuity.
\end{proof}

By Proposition~\ref{prop scalar lower bound}, the hypothesis of
        Proposition~\ref{prop scalar comparison} follows from
        $\mathcal R_{g(0)}\ge-K_*$.
The scalar curvature average has bounded variation and is
        differentiable almost everywhere, absolute continuity is not asserted.

On a closed Ricci flow, the corresponding classical identity is
\[
    \frac{d}{dt}\int_M\mathcal R u\,d\mu_{g(t)}
    =2\int_M|\Ric|^2u\,d\mu_{g(t)},
\]
see \cite[equations~(3-10) and~(4-19)]{GPT13PJM}.
For the corresponding scalar heat kernel representation under complete
        Ricci flow with bounded curvature, see \cite[Chapter~26.1]{CGIKLN10Book}
        and \cite[Section~3.2, equation~(3-14)]{LNN23GT} for example.

\subsection{An auxiliary equation for vector fields}

In this subsection, we will investigate weighted integral identities about an auxiliary vector field, which will be used in
        Section~\ref{sec:entropy} to establish energy monotonicity.

For vector fields $Y$ and scalar functions $f$, write
\[
    \|Y\|_{\infty,s}:=\sup_M|Y|_{g(s)},
    \qquad \|f\|_\infty:=\sup_M|f|.
\]

The evolution equation \eqref{eq vector heat} below is the vector-field form of
        the Hodge heat equation used by Ilmanen-Knopf
        \cite[Section~2]{IK03MRL}, it also appears in Brendle's study of
        rotational symmetry \cite[Section~5]{Brendle20Acta}.

\begin{lemma}
    \label{lem vector existence}
    Let $0\le s<T$ and $X_s\in C_c^\infty(TM)$.
    There is a smooth solution of
    \begin{equation}
        \label{eq vector heat}
        \partial_tX=\Delta_{g(t)}X+\Ric_{g(t)}(X),\qquad X(s)=X_s,
    \end{equation}
    on $M\times[s,T]$ satisfying
    \begin{equation}
        \label{eq vector bound}
        |X(t)|_{g(t)}\le\|X_s\|_{\infty,s}.
    \end{equation}
    Moreover,
    \begin{align}
        \label{eq vector norm}
        \Box|X|_{g(t)}^2&=-2|\nabla X|_{g(t)}^2,\\
        \label{eq vector divergence}
        \Box\diver_{g(t)}X&=2\langle\Ric,\nabla X\rangle_{g(t)}.
    \end{align}
\end{lemma}

\begin{proof}
    For any smooth solution of \eqref{eq vector heat},
    \begin{align*}
        \partial_t|X|^2
        &=-2\Ric(X,X)+2\langle\partial_tX,X\rangle
            =2\langle\Delta X,X\rangle,\\
        \Delta|X|^2
        &=2\langle\Delta X,X\rangle+2|\nabla X|^2.
    \end{align*}
    Subtracting these identities, we obtain \eqref{eq vector norm}.

    If $M$ is compact, the solution exists by standard linear parabolic theory.
    Otherwise, solve \eqref{eq vector heat} with zero Dirichlet boundary
            values on a smooth exhaustion $\{\Omega_j\}$ of $M$ with
            $\supp X_s\subset\subset\Omega_1$.
    By \eqref{eq vector norm} and the scalar maximum principle, these
            solutions satisfy \eqref{eq vector bound} uniformly in $j$.
    For each fixed $k$ and $j\ge k+1$, the system has smooth coefficients
            and is uniformly parabolic on $\Omega_{k+1}\times[s,T]$.
    By local parabolic Schauder estimates for this system, including
            estimates at $t=s$ with the common smooth initial data $X_s$,
            we obtain bounds of every order on
            $\overline{\Omega}_k\times[s,T]$ independent of $j$.
    Therefore, we can use the Arzel\`a-Ascoli theorem and a standard diagonal
            argument to extract a subsequence converging locally smoothly
            to a solution of \eqref{eq vector heat} on $M\times[s,T]$.

    To prove \eqref{eq vector divergence}, write $X^\flat=g(t)(X,\cdot)$
            and let $\delta$ be the formal adjoint of $d$
            with respect to $g(t)$.
    By the Bochner-Weitzenb\"ock formula for one-forms, we have
    \[
        (\Delta X-\Ric(X))^\flat=-(d\delta+\delta d)X^\flat,
    \]
    see \cite[p.~17 and Exercise~1.43, p.~20]{CLN06Book}.
    Since $\delta X^\flat=-\diver X$, $\delta^2=0$, and
            $\delta df=-\Delta f$, we have
    \[
        \diver(\Delta X-\Ric(X))
        =\delta(d\delta+\delta d)X^\flat
        =\delta d(\delta X^\flat)=\Delta\diver X.
    \]
    In local coordinates, we have
    \[
        \partial_t\Gamma^i_{ij}
        =-g^{i\ell}\bigl(\nabla_iR_{j\ell}
                +\nabla_jR_{i\ell}-\nabla_\ell R_{ij}\bigr)
        =-\nabla_j\mathcal R.
    \]
    Consequently,
    \begin{align*}
        \partial_t\diver X
        &=\partial_t\bigl(\partial_iX^i+\Gamma^i_{ij}X^j\bigr)\\
        &=\diver(\partial_tX)+(\partial_t\Gamma^i_{ij})X^j\\
        &=\diver(\partial_tX)-\langle\grad\mathcal R,X\rangle.
    \end{align*}
    By the contracted Bianchi identity, we have
    \begin{align*}
        \diver(\Ric(X))
        &=\nabla_i\bigl(R^i{}_jX^j\bigr)\\
        &=(\nabla^iR_{ij})X^j+R_{ij}\nabla^iX^j\\
        &=\frac12\langle\grad\mathcal R,X\rangle
            +\langle\Ric,\nabla X\rangle.
    \end{align*}
    Combining these identities with \eqref{eq vector heat}, we obtain
    \begin{align*}
        \Box\diver X
        &=\diver(\Delta X+\Ric(X))-\Delta\diver X
            -\langle\grad\mathcal R,X\rangle\\
        &=2\diver(\Ric(X))-\langle\grad\mathcal R,X\rangle\\
        &=2\langle\Ric,\nabla X\rangle.
    \end{align*}
    This completes the proof.
\end{proof}

\begin{lemma}
    \label{lem divergence bound}
    Let $X$ be a solution given by Lemma~\ref{lem vector existence}.
    Assume that $\mathcal R_{g(t)}\ge-K_*$ on $M\times[s,T]$ for some
            $K_*\ge0$.
    For every $\lambda>0$, on $M\times[s,T]$ we have
    \begin{equation}
        \label{eq divergence strong}
        \begin{aligned}
            |\diver_{g(t)}X|
            &\le\|\diver_{g(s)}X_s\|_\infty
                +\lambda(\mathcal R_{g(t)}+K_*)\\
            &\qquad+\frac{\|X_s\|_{\infty,s}^2-|X|_{g(t)}^2}{\lambda}.
        \end{aligned}
    \end{equation}
    Consequently,
    \begin{equation}
        \label{eq divergence bound}
        |\diver_{g(t)}X|\le\|\diver_{g(s)}X_s\|_\infty
            +2\|X_s\|_{\infty,s}\sqrt{\mathcal R_{g(t)}+K_*}.
    \end{equation}
\end{lemma}

\begin{proof}
    Fix $\lambda>0$ and set
    \[
        w:=\diver X-\lambda(\mathcal R+K_*)
            -\frac{\|X_s\|_{\infty,s}^2-|X|^2}{\lambda}\le\diver X.
    \]
    By \eqref{eq vector norm}, \eqref{eq vector divergence}, and
            $\Box\mathcal R=2|\Ric|^2$, we obtain
    \begin{align*}
        \Box w
        &=2\langle\Ric,\nabla X\rangle
            -2\lambda|\Ric|^2-2\lambda^{-1}|\nabla X|^2\\
        &\le-\lambda|\Ric|^2-\lambda^{-1}|\nabla X|^2
        \le-\frac{(\diver X)^2}{n\lambda},
    \end{align*}
    where we used $|\diver X|^2=|\tr\nabla X|^2\le n|\nabla X|^2$
            in the last step.
    Hence $\Box w\le-w^2/(n\lambda)$ wherever $w>0$.
    Since $w(s)\le\|\diver_{g(s)}X_s\|_\infty$,
            by Proposition~\ref{prop quadratic maximum}, we obtain
            $w\le\|\diver_{g(s)}X_s\|_\infty$.
    By the same argument for $-X$, we obtain \eqref{eq divergence strong}.
    We omit the term $-|X|^2/\lambda$ and take the infimum over
            $\lambda>0$ to obtain \eqref{eq divergence bound}.
    This completes the proof.
\end{proof}

\begin{proposition}
    \label{prop exact pairings}
    Let $X$ be a solution given by Lemma~\ref{lem vector existence}.
    The function $t\mapsto[|X|^2]_t$ is absolutely continuous on $[s,T]$,
            and, for every $s\le t_1<t_2\le T$,
    \begin{equation}
        \label{eq norm pairing}
        [|X|^2]_{t_2}-[|X|^2]_{t_1}
        =-2\int_{t_1}^{t_2}[|\nabla X|^2]_t\,dt.
    \end{equation}
    If we further assume that $\mathcal R_{g(t)}\ge-K_*$ on $M\times[s,T]$ for
            some $K_*\ge0$, then $\diver_{g(t)}X(t)\in L^1(M,d\nu_t)$ for
            every $t\in[s,T]$, the function $t\mapsto[\diver X]_t$ is
            absolutely continuous on $[s,T]$, and
    \begin{equation}
        \label{eq divergence pairing}
        [\diver X]_{t_2}-[\diver X]_{t_1}
        =2\int_{t_1}^{t_2}
            \bigl[\langle\Ric,\nabla X\rangle\bigr]_t\,dt.
    \end{equation}
\end{proposition}

\begin{proof}
    Fix $s\le t_1<t_2\le T$.
    Consider the exhaustion $\{\Omega_j\}$ and Dirichlet solutions $u_j$
            constructed above with terminal value $u_j(T)=\phi^2$.
    Each $u_j$ is smooth on $\overline{\Omega}_j\times[0,T]$.
    The zero extensions of $u_j$ increase to $u$.
    By monotone convergence and \eqref{eq conjugate equation},
    \[
        \int_{\Omega_j}u_j(t)\,d\mu_{g(t)}\to1
        \quad\text{for every}\quad t\in[s,T].
    \]
    Let $V$ be the outward unit normal with respect to $g(t)$ and let
            $d\sigma_{g(t)}$ be the induced boundary measure.
    Since $u_j\ge0$ and $u_j=0$ on $\partial\Omega_j$, we have
            $\partial_Vu_j\le0$.
    By the conjugate heat equation, \eqref{eq vector norm}, and
            Green's formula, we obtain
    \begin{equation}
        \label{eqn ddt X L2}
        \begin{aligned}
            \frac{d}{dt}\int_{\Omega_j}u_j|X|^2\,d\mu_{g(t)}
            &=-2\int_{\Omega_j}u_j|\nabla X|^2\,d\mu_{g(t)}\\
            &\qquad-\int_{\partial\Omega_j}|X|^2\partial_Vu_j
                    \,d\sigma_{g(t)}.
        \end{aligned}
    \end{equation}
    Similarly,
    \[
        \frac{d}{dt}\int_{\Omega_j}u_j\,d\mu_{g(t)}
        =-\int_{\partial\Omega_j}\partial_Vu_j\,d\sigma_{g(t)}.
    \]
    By this identity and \eqref{eq vector bound},
    \begin{align*}
        0&\le-\int_{t_1}^{t_2}\int_{\partial\Omega_j}
                |X|^2\partial_Vu_j\,d\sigma_{g(t)}\,dt\\
        &\le\|X_s\|_{\infty,s}^2
                \left(\int_{\Omega_j}u_j(t_2)\,d\mu_{g(t_2)}
                -\int_{\Omega_j}u_j(t_1)\,d\mu_{g(t_1)}\right)
        \to0.
    \end{align*}
    Integrate \eqref{eqn ddt X L2} over $[t_1,t_2]$.
    By monotone convergence, we obtain
    \[
        [|X|^2]_{t_2}+2\int_{t_1}^{t_2}[|\nabla X|^2]_t\,dt
        =[|X|^2]_{t_1}\le\|X_s\|_{\infty,s}^2<+\infty.
    \]
    This proves \eqref{eq norm pairing} and the integrability of
            $t\mapsto[|\nabla X|^2]_t$ on $[s,T]$.
    Taking $t_1=s$, we conclude that
            $t\mapsto[|X|^2]_t$ is absolutely continuous.
    In particular,
    \begin{equation}
        \label{eq vector energy}
        \int_s^t[|\nabla X|^2]_\tau\,d\tau
        \le\frac12[|X_s|^2]_s
        \quad\text{for all}\quad t\in[s,T].
    \end{equation}

    Now assume that $\mathcal R_{g(t)}\ge-K_*$ on $M\times[s,T]$.
    For every $t\in[s,T]$, by Lemma~\ref{lem divergence bound}
            and Cauchy-Schwarz, we have
    \begin{align*}
        [|\diver X|]_t
        &\le\|\diver_{g(s)}X_s\|_\infty
            +2\|X_s\|_{\infty,s}[(\mathcal R+K_*)^{1/2}]_t\\
        &\le\|\diver_{g(s)}X_s\|_\infty
            +2\|X_s\|_{\infty,s}
                \bigl([\mathcal R]_t+K_*\bigr)^{1/2}
        <+\infty,
    \end{align*}
    where finiteness follows from Proposition~\ref{prop scalar comparison}.
    Moreover, we also have
    \begin{align*}
        \int_s^T[|\langle\Ric,\nabla X\rangle|]_t\,dt
        &\le\int_s^T[|\Ric|^2]_t^{1/2}
                [|\nabla X|^2]_t^{1/2}\,dt\\
        &\le\left(\int_s^T[|\Ric|^2]_t\,dt\right)^{1/2}
            \left(\int_s^T[|\nabla X|^2]_t\,dt\right)^{1/2}
        <+\infty,
    \end{align*}
    where finiteness follows from \eqref{eq ricci integrable}
            and \eqref{eq vector energy}.
    Hence $[\langle\Ric,\nabla X\rangle]_t$ is absolutely convergent
            for almost every $t$, and
    \[
        \int_s^T\bigl|[\langle\Ric,\nabla X\rangle]_t\bigr|\,dt
        \le\int_s^T[|\langle\Ric,\nabla X\rangle|]_t\,dt
        <+\infty.
    \]

    Fix $s\le t_1<t_2\le T$.
    To prove \eqref{eq divergence pairing}, we apply
            Lemma~\ref{lem positive comparison} to nonnegative functions
            involving both signs of $\diver X$.
    For $\lambda>0$, set
    \[
        v_\lambda^\pm
        :=\|\diver_{g(s)}X_s\|_\infty+\lambda(\mathcal R+K_*)
            +\frac{\|X_s\|_{\infty,s}^2-|X|^2}{\lambda}
            \pm\diver X.
    \]
    By \eqref{eq divergence strong}, we have $v_\lambda^\pm\ge0$.
    By \eqref{eq vector norm}, \eqref{eq vector divergence},
            and $\Box\mathcal R=2|\Ric|^2$, we also have
    \begin{align*}
        \Box v_\lambda^\pm
        &=2\lambda|\Ric|^2+2\lambda^{-1}|\nabla X|^2
            \pm2\langle\Ric,\nabla X\rangle\\
        &\ge\lambda|\Ric|^2+\lambda^{-1}|\nabla X|^2\ge0.
    \end{align*}
    By the above estimates,
    \[
        [v_\lambda^\pm]_{t_i}<+\infty\quad(i=1,2),
        \qquad
        \int_{t_1}^{t_2}[\Box v_\lambda^\pm]_t\,dt<+\infty.
    \]
    By Lemma~\ref{lem positive comparison}, we obtain
    \begin{align*}
        0&\le[v_\lambda^\pm]_{t_2}-[v_\lambda^\pm]_{t_1}
            -\int_{t_1}^{t_2}[\Box v_\lambda^\pm]_t\,dt\\
        &=\lambda\left(
            [\mathcal R]_{t_2}-[\mathcal R]_{t_1}
            -2\int_{t_1}^{t_2}[|\Ric|^2]_t\,dt\right)\\
        &\qquad-\lambda^{-1}\left(
            [|X|^2]_{t_2}-[|X|^2]_{t_1}
            +2\int_{t_1}^{t_2}[|\nabla X|^2]_t\,dt\right)\\
        &\qquad\pm\left(
            [\diver X]_{t_2}-[\diver X]_{t_1}
            -2\int_{t_1}^{t_2}
                [\langle\Ric,\nabla X\rangle]_t\,dt\right).
    \end{align*}
    By \eqref{eq norm pairing}, we therefore obtain
    \[
        \begin{aligned}
            &\left|[\diver X]_{t_2}-[\diver X]_{t_1}
                -2\int_{t_1}^{t_2}
                    [\langle\Ric,\nabla X\rangle]_t\,dt\right|\\
            &\qquad\le\lambda\left(
                [\mathcal R]_{t_2}-[\mathcal R]_{t_1}
                -2\int_{t_1}^{t_2}[|\Ric|^2]_t\,dt\right).
        \end{aligned}
    \]
    By Proposition~\ref{prop scalar comparison}, the right hand side
            is nonnegative and finite.
    Letting $\lambda\to0^+$, we obtain \eqref{eq divergence pairing}.
    In particular, for every $t\in[s,T]$,
    \[
        [\diver X]_t=[\diver X]_s
            +2\int_s^t[\langle\Ric,\nabla X\rangle]_r\,dr.
    \]
    Since the integrand belongs to $L^1([s,T])$, the function
            $t\mapsto[\diver X]_t$ is absolutely continuous.
    This completes the proof.
\end{proof}

Under the scalar curvature lower bound in Proposition~\ref{prop exact pairings},
        define
\[
    J_X(t):=[\mathcal R+2\diver X-|X|^2]_t.
\]
Combining \eqref{eq scalar energy}, \eqref{eq norm pairing}, and
        \eqref{eq divergence pairing}, we obtain, for $s\le t_1<t_2\le T$,
\begin{equation}
    \label{eq vector fisher}
    \begin{split}
        J_X(t_2)-J_X(t_1)
        &\ge2\int_{t_1}^{t_2}
            \bigl[|\Ric|^2+2\langle\Ric,\nabla X\rangle
                +|\nabla X|^2\bigr]_t\,dt\\
        &=2\int_{t_1}^{t_2}[|\Ric+\nabla X|^2]_t\,dt.
    \end{split}
\end{equation}
This is the vector-field counterpart of Perelman's $\mathcal F$-monotonicity
        formula \cite[equation~(1.4)]{Perelman02arXiv}.
$J_X$ will be used in Section~\ref{sec:entropy} to establish the monotonicity for
        Perelman's energy and entropy functionals.

\section{Perelman's energy and entropy}
\label{sec:entropy}

Let $(M^n,g(t))$ be a complete Ricci flow for $t\in[0,T]$,
        and let $u$ be given by \eqref{eq conjugate definition}, with
        $\phi\in C_c^\infty(M)$, $\phi\ge0$, and $\int_M\phi^2\,d\mu_{g(T)}=1$.
We use the weighted averages defined in Section~\ref{sec:preliminaries}.

\subsection{Perelman's energy}

Suppose that $\mathcal R_{g(t)}\ge-K_*$ on $M\times[0,T]$ for some $K_*\ge0$.
For $0\le t<T$, define
\[
    I(t):=[|\nabla\log u|^2]_t
        =\int_M\frac{|\nabla u(t)|_{g(t)}^2}{u(t)}\,d\mu_{g(t)},
    \qquad F(t):=[\mathcal R]_t+I(t).
\]
We have $[\mathcal R]_t<+\infty$ by Proposition~\ref{prop scalar comparison},
        while $I(t)$ is allowed to equal $+\infty$ before we do the analysis.
The quantity $I(t)$ is the Fisher information of $u(t)$ relative to
        $d\mu_{g(t)}$, see \cite[p.~19233]{CM13PNAS}.
We also have
\[
    F(t)=\mathcal F(g(t),-\log u(t)),\qquad
    \mathcal F(g,f):=\int_M(\mathcal R_g+|\nabla f|_g^2)e^{-f}\,d\mu_g,
\]
where $\mathcal F$ is Perelman's energy functional
        \cite[Section~1.1]{Perelman02arXiv}.
At the terminal time, set
\[
    I(T):=4\int_M|\nabla\phi|_{g(T)}^2\,d\mu_{g(T)}
\]
and
\begin{equation}
    \label{eq terminal fisher}
    F(T):=[\mathcal R]_T+I(T)
        =\int_M\left(\mathcal R_{g(T)}\phi^2+4|\nabla\phi|_{g(T)}^2\right)
            \,d\mu_{g(T)}<+\infty.
\end{equation}

\begin{proposition}
    \label{prop fisher finite}
    Assume that $\mathcal R_{g(t)}\ge-K_*$ on $M\times[0,T]$ for
            some $K_*\ge0$.
    Then, for every $t\in[0,T]$,
    \[
        -K_*\le F(t)\le F(T),\qquad 0\le I(t)\le F(T)+K_*.
    \]
    If $X$ is a solution on $M\times[s,T]$ given by
            Lemma~\ref{lem vector existence}, then $J_X(t)\le F(t)$
            for every $t\in[s,T]$, and
    \begin{equation}
        \label{eq fisher representation}
        J_X(t)=F(t)-[|X+\grad\log u|^2]_t
        \quad\text{for}\quad s\le t<T.
    \end{equation}
    Moreover,
    \begin{equation}
        \label{eq terminal energy continuity}
        \lim\limits_{t\to T^-}I(t)=I(T),
        \qquad\lim\limits_{t\to T^-}F(t)=F(T).
    \end{equation}
\end{proposition}

\begin{proof}
    Fix $0\le s<T$.
    We first prove the variational formula
    \begin{equation}
        \label{eq fisher variational}
        F(s)=[\mathcal R]_s+
            \sup\limits_{Z\in C_c^\infty(TM)}
            \left(2[\diver Z]_s-[|Z|^2]_s\right).
    \end{equation}
    For every $Z\in C_c^\infty(TM)$, by integration by parts, it holds that
    \[
        [\diver Z]_s
        =-\int_M\langle Z,\grad u(s)\rangle\,d\mu_{g(s)}
        =-[\langle Z,\grad\log u(s)\rangle]_s.
    \]
    It follows that
    \[
        2[\diver Z]_s-[|Z|^2]_s
        =[-2\langle Z,\grad\log u(s)\rangle-|Z|^2]_s
        \le [|\nabla\log u|^2]_s=I(s).
    \]
    Taking the supremum over $Z$, we obtain
    \[
        \sup\limits_{Z\in C_c^\infty(TM)}
            \left(2[\diver Z]_s-[|Z|^2]_s\right)\le I(s).
    \]
    Conversely, choose $\zeta_j\in C_c^\infty(M)$ with
            $0\le\zeta_j\le\zeta_{j+1}\le1$ and $\zeta_j\to1$ on $M$,
            and set $Z_j:=-\zeta_j\grad\log u(s)$.
    By monotone convergence,
    \[
        2[\diver Z_j]_s-[|Z_j|^2]_s
        =[(2\zeta_j-\zeta_j^2)|\nabla\log u(s)|^2]_s
        \to I(s).
    \]
    This proves \eqref{eq fisher variational}.

    We then prove $F(s)\le F(T)$.
    Fix any $Z\in C_c^\infty(TM)$, and let $Y$ be a solution
            given by Lemma~\ref{lem vector existence} with $Y(s)=Z$.
    Since $\phi$ has compact support, we have
    \begin{align*}
        J_Y(T)
        &=\int_M\left(\mathcal R\phi^2
                -4\phi\langle Y,\grad\phi\rangle-|Y|^2\phi^2\right)
                \,d\mu_{g(T)}\\
        &=F(T)-\int_M|\phi Y(T)+2\grad\phi|^2\,d\mu_{g(T)}
        \le F(T).
    \end{align*}
    By \eqref{eq vector fisher} and $Y(s)=Z$, we therefore have
    \[
        [\mathcal R]_s+2[\diver Z]_s-[|Z|^2]_s
        =J_Y(s)\le J_Y(T)\le F(T).
    \]
    Since $Z$ is arbitrary, taking the supremum and using
            \eqref{eq fisher variational}, we obtain $F(s)\le F(T)$.
    By $I(s)\ge0$, $F(s)=[\mathcal R]_s+I(s)$,
            and $[\mathcal R]_s\ge-K_*$, it follows that
    \[
        -K_*\le F(s)\le F(T),
        \qquad
        0\le I(s)=F(s)-[\mathcal R]_s\le F(T)+K_*.
    \]
    These bounds also hold at $s=T$, since $I(T)\ge0$,
            $F(T)=[\mathcal R]_T+I(T)$, and $[\mathcal R]_T\ge-K_*$.

    We now prove \eqref{eq fisher representation}.
    By Proposition~\ref{prop exact pairings}, $[|\diver X|]_t<+\infty$.
    By \eqref{eq vector bound}, the above bound for $I(t)$ and Cauchy-Schwarz,
            we have
    \[
        [|\langle X,\grad\log u\rangle|]_t
        \le\|X_s\|_{\infty,s}[|\nabla\log u|]_t
        \le\|X_s\|_{\infty,s}\sqrt{I(t)}<+\infty.
    \]
    Since $X(t)$ need not have compact support, we justify integration
            by parts using cutoffs.
    For each $L>0$, choose a compactly supported Lipschitz cutoff
            $\varphi_L$ with $0\le\varphi_L\le1$,
            $|\nabla\varphi_L|\le2/L$ almost everywhere,
            and $\varphi_L\to1$ pointwise as $L\to+\infty$.
    Then we have
    \[
        [\varphi_L\diver X]_t=-[\varphi_L\langle X,\grad\log u\rangle]_t
                -[\langle X,\grad\varphi_L\rangle]_t,
    \]
    with
    \[
        \bigl|[\langle X,\grad\varphi_L\rangle]_t\bigr|
                \le\frac{2\|X_s\|_{\infty,s}}{L}.
    \]
    Letting $L\to+\infty$, by dominated convergence, we obtain
    \begin{equation}
        \label{eq weighted divergence}
        [\diver X]_t=-[\langle X,\grad\log u\rangle]_t.
    \end{equation}
    Consequently, by the definition of $J_X$,
    \begin{align*}
        J_X(t)&=[\mathcal R]_t
                -2[\langle X,\grad\log u\rangle]_t-[|X|^2]_t\\
            &=F(t)-[|X+\grad\log u|^2]_t\le F(t).
    \end{align*}

    It remains to prove \eqref{eq terminal energy continuity}.
    By local smooth convergence, on $\{\phi>0\}$ we have
    \[
        \lim_{t\to T^-}
            \left(\frac{|\nabla u(t)|_{g(t)}^2}{u(t)}
                \frac{d\mu_{g(t)}}{d\mu_{g(T)}}\right)
        =4|\nabla\phi|_{g(T)}^2.
    \]
    On $\{\phi=0\}$, smoothness and nonnegativity of $\phi$ imply
            $\nabla\phi=0$.
    Since the integrands are nonnegative, we have
    \[
        4|\nabla\phi|_{g(T)}^2=0
        \le\liminf_{t\to T^-}
            \left(\frac{|\nabla u(t)|_{g(t)}^2}{u(t)}
                \frac{d\mu_{g(t)}}{d\mu_{g(T)}}\right).
    \]
    By Fatou's lemma with respect to $d\mu_{g(T)}$, we obtain
    \[
        I(T)=4\int_M|\nabla\phi|_{g(T)}^2\,d\mu_{g(T)}
        \le\liminf_{t\to T^-}I(t).
    \]
    On the other hand, since $I(t)\le F(T)-[\mathcal R]_t$ and $[\mathcal R]_t$
            is left continuous by Proposition~\ref{prop scalar comparison},
            it follows that
    \[
        \limsup_{t\to T^-}I(t)
        \le F(T)-[\mathcal R]_T=I(T).
    \]
    Hence $I(t)\to I(T)$.
    Since $F(t)=[\mathcal R]_t+I(t)$, the same left continuity
            also implies $F(t)\to F(T)$.
\end{proof}

On a closed Ricci flow with a positive conjugate probability density
        $u$, Perelman's formula is
\[
    F'(t)=2\int_M|\Ric-\nabla^2\log u|^2u\,d\mu_{g(t)}
        \ge\frac2nF(t)^2,
\]
see \cite[Section~1.2, equation~(1.4) and the following
        inequality]{Perelman02arXiv}.
The following proposition establishes the differential inequality
        needed here for complete flows under a scalar curvature lower bound.

\begin{proposition}
    \label{prop fisher differential}
    Assume that $\mathcal R_{g(t)}\ge-K_*$ on $M\times[0,T]$ for
            some $K_*\ge0$.
    Then $F$ is nondecreasing on $[0,T]$ and satisfies
    \begin{equation}
        \label{eq fisher differential}
        dF\ge\frac2nF^2\,dt
    \end{equation}
    as an inequality of measures on $(0,T)$.
\end{proposition}

Here $dF$ denotes the nonnegative Radon measure representing the
        distributional derivative of $F$, and $dt$ is Lebesgue measure.

\begin{proof}
    Fix $0\le s<T$ and choose $\zeta_k\in C_c^\infty(M)$ with
            $0\le\zeta_k\le\zeta_{k+1}\le1$ and $\zeta_k\to1$ on $M$.
    Let $X_k$ be a solution from Lemma~\ref{lem vector existence} with
            $X_k(s)=-\zeta_k\grad\log u(s)$.
    By the same calculation as in the proof of
            \eqref{eq fisher variational}, we obtain
    \begin{equation}
        \label{eq score approximation}
            [\diver X_k]_s\to I(s),\qquad
            [|X_k|^2]_s\le I(s),\qquad
            J_{X_k}(s)\to F(s).
    \end{equation}
    For $s<t\le T$, Proposition~\ref{prop fisher finite} and
            \eqref{eq vector fisher} imply
    \[
        F(t)\ge J_{X_k}(t)\ge J_{X_k}(s).
    \]
    Letting $k\to\infty$, we obtain $F(t)\ge F(s)$.

    To prove \eqref{eq fisher differential}, fix $s\in(0,T)$ where
            $F$ is differentiable.
    For $0<h<T-s$, by \eqref{eq score approximation}, choose one of
            the above solutions, denoted by $X_h$, such that
    \[
        0\le F(s)-J_{X_h}(s)\le h^2.
    \]
    For any $s\le\tau\le s+h$, by \eqref{eq fisher representation}
            and monotonicity of $F$ and $J_{X_h}$, we have
    \[
        [|X_h+\grad\log u|^2]_\tau
        =F(\tau)-J_{X_h}(\tau)
        \le F(s+h)-F(s)+h^2.
    \]
    By \eqref{eq weighted divergence} and Cauchy-Schwarz,
    \begin{align*}
        \bigl|[\mathcal R+\diver X_h]_\tau-F(\tau)\bigr|
        &=\bigl|[\langle X_h+\grad\log u,\grad\log u\rangle]_\tau\bigr|\\
        &\le\sqrt{I(\tau)[|X_h+\grad\log u|^2]_\tau}\\
        &\le\sqrt{(F(T)+K_*)\bigl(F(s+h)-F(s)+h^2\bigr)}.
    \end{align*}
    Since $F$ is continuous at $s$, we obtain
    \[
        \sup_{s\le\tau\le s+h}
        \bigl|[\mathcal R+\diver X_h]_\tau-F(s)\bigr|\to0
        \quad\text{as}\quad h\to0^+.
    \]
    By \eqref{eq vector fisher}, the trace inequality, and Cauchy-Schwarz,
            we have
    \begin{align*}
        F(s+h)-F(s)+h^2
        &\ge J_{X_h}(s+h)-J_{X_h}(s)\\
        &\ge2\int_s^{s+h}[|\Ric+\nabla X_h|^2]_\tau\,d\tau\\
        &\ge\frac2n\int_s^{s+h}[\mathcal R+\diver X_h]_\tau^2\,d\tau.
    \end{align*}
    The integral is finite by \eqref{eq ricci integrable} and
            \eqref{eq vector energy}.
    Dividing by $h$ and letting $h\to0^+$, we obtain
    \[
        F'(s)\ge\frac2nF(s)^2.
    \]
    Since $F$ is nondecreasing, the singular part of $dF$ is nonnegative.
    Therefore,
    \[
        dF\ge F'\,dt\ge\frac2nF^2\,dt
        \quad\text{on}\quad(0,T),
    \]
    which proves \eqref{eq fisher differential}.
\end{proof}

\subsection{Boltzmann-Shannon entropy}

We consider the Boltzmann-Shannon entropy $H=-\int_Mu\log u\,d\mu_{g(t)}$
        relative to the evolving Riemannian volume.
For closed Ricci flows, the identity $H'=-F$ is classical, see
        \cite[equation~(6.5.1)]{Topping06Book} and
        \cite[Theorem~2.1, equation~(2-7)]{GPT13PJM}. Since the Boltzmann-Shannon entropy $H$ is closely related to the Nash entropy, one may also compare  related results for pointed Nash entropy in \cite{HN14CPAM,Bamler20arXiv}.
 We will prove finiteness and the integral identity for $H$ under only a scalar
        curvature lower bound.
Since $F$ may be negative, $H$ need not be monotone.

We first justify integration by parts for bounded functions using
        solutions of the Dirichlet heat equation as cutoffs.
The boundary terms will be controlled by conservation of
        mass.

\begin{lemma}
    \label{lem bounded chain pairing}
    Let $0\le s<t\le T$, and let $v$ be a bounded smooth function
            on $M\times[s,t]$.
    For $s\le\tau<t$, define
    \[
        q:=\partial_\tau v+\Delta_{g(\tau)}v
            +2\langle\nabla\log u,\nabla v\rangle_{g(\tau)},
    \]
    where the spatial operators are taken with respect to $g(\tau)$.
    If $\int_s^t[|q|]_\tau\,d\tau<+\infty$, then
    \begin{equation}
        \label{eq bounded chain pairing}
        [v]_t-[v]_s=\int_s^t[q]_\tau\,d\tau.
    \end{equation}
\end{lemma}

\begin{proof}
    By the definition of $q$ and the identity $\Box^*u=0$, we have
    \[
        \Box^*(uv)=-u(\partial_\tau v+\Delta v)
            -2\langle\nabla u,\nabla v\rangle=-uq.
    \]
    When $t=T$, the identity extends locally smoothly to $T$.

    Choose a smooth exhaustion $\{\Omega_j\}$ of $M$ and
            $\eta_j\in C_c^\infty(\Omega_j)$ such that
            $0\le\eta_j\le\eta_{j+1}\le1$ and $\eta_j\to1$ on $M$.
    Let $h_j$ solve the heat equation on $\Omega_j\times[s,t]$
            with initial value $\eta_j$ and zero boundary values,
            and extend $h_j$ by zero outside $\Omega_j$.
    If $M$ is compact, take $\Omega_j=M$ and $\eta_j=h_j\equiv1$.
    By parabolic maximum principle, $0\le h_j\le1$.
    It follows from Dirichlet kernel exhaustion and
            \eqref{eq heat conservation} that $h_j\to1$ pointwise
            on $M\times[s,t]$.

    Let $V$ denote the outward unit normal to $\partial\Omega_j$ and
            $d\sigma$ its induced boundary measure.
    Since $h_j\ge0$ in $\Omega_j$ and $h_j=0$ on $\partial\Omega_j$, $\partial_Vh_j\le0$.
    Using $\partial_\tau h_j=\Delta h_j$,
            $\partial_\tau d\mu=-\mathcal R\,d\mu$ and
            $\Box^*(uv)=-uq$, we obtain
    \begin{equation}
        \label{eq chain local green}
        \begin{aligned}
            \frac{d}{d\tau}\int_{\Omega_j}h_juv\,d\mu
            &=\int_{\Omega_j}\bigl(uv\Delta h_j+h_j\partial_\tau(uv)
                    -\mathcal Rh_juv\bigr)\,d\mu\\
            &=\int_{\Omega_j}\bigl(uv\Delta h_j-h_j\Delta(uv)
                    +h_juq\bigr)\,d\mu\\
            &=\int_{\Omega_j}h_juq\,d\mu
                +\int_{\partial\Omega_j}uv\partial_Vh_j\,d\sigma.
        \end{aligned}
    \end{equation}
    Taking $v=1$ and integrating over $(s,t)$, we obtain
    \begin{align*}
        0\le{}&\int_s^t\int_{\partial\Omega_j}
                u(-\partial_Vh_j)\,d\sigma_{g(\tau)}\,d\tau\\
        ={}&\int_{\Omega_j}\eta_ju(s)\,d\mu_{g(s)}
                -\int_{\Omega_j}h_j(t)u(t)\,d\mu_{g(t)}\to0.
    \end{align*}
    Here both integrals tend to $1$ by dominated convergence
            and \eqref{eq conjugate equation}.
    For general $v$, it follows that
    \[
        \begin{aligned}
            &\left|\int_s^t\int_{\partial\Omega_j}
                uv\partial_Vh_j\,d\sigma_{g(\tau)}\,d\tau\right|\\
            &\qquad\le\sup_{M\times[s,t]}|v|
                \int_s^t\int_{\partial\Omega_j}
                    u(-\partial_Vh_j)\,d\sigma_{g(\tau)}\,d\tau\to0.
        \end{aligned}
    \]

    Finally, since $0\le h_j\le1$, we have
    \[
        |h_juv|\le\sup_{M\times[s,t]}|v|\cdot u
        \quad\text{and}\quad
        |h_juq|\le u|q|.
    \]
    Integrating \eqref{eq chain local green} over $(s,t)$ and letting
            $j\to\infty$, we obtain \eqref{eq bounded chain pairing}
            by dominated convergence.
    This completes the proof.
\end{proof}

\begin{proposition}
    \label{prop entropy dissipation}
    Assume that $\mathcal R_{g(t)}\ge-K_*$ on $M\times[0,T]$ for
            some $K_*\ge0$.
    Then $u(t)\log u(t)\in L^1(M,d\mu_{g(t)})$ for every $t\in[0,T]$.
    Moreover, the function
    \[
        H(t):=-\int_Mu(t)\log u(t)\,d\mu_{g(t)}
    \]
    is absolutely continuous on $[0,T]$ and satisfies
    \begin{equation}
        \label{eq entropy dissipation}
        H(t)-H(s)=-\int_s^tF(\tau)\,d\tau
        \quad\text{for all}\quad0\le s<t\le T.
    \end{equation}
\end{proposition}

\begin{proof}
    Let $U>0$ be the upper bound in \eqref{eq conjugate upper}.
    For any $\varepsilon>0$, set
    \[
        v_\varepsilon:=\log\frac{U+\varepsilon}{u+\varepsilon}.
    \]
    The function $v_\varepsilon$ is locally smooth up to $T$, nonnegative
            and bounded.
    Wherever $u>0$, it holds that $v_\varepsilon\to\log(U/u)$ as
            $\varepsilon\to0^+$.

    For $\tau<T$, by the conjugate heat equation \eqref{eq conjugate equation},
    \[
        \partial_\tau v_\varepsilon
            =\frac{\Delta u-\mathcal Ru}{u+\varepsilon},
        \qquad
        \Delta v_\varepsilon
            =-\frac{\Delta u}{u+\varepsilon}
                +\frac{|\nabla u|^2}{(u+\varepsilon)^2}.
    \]
    Then
    \[
        \begin{aligned}
            q_\varepsilon:={}&\partial_\tau v_\varepsilon+\Delta v_\varepsilon
                +2\langle\nabla\log u,\nabla v_\varepsilon\rangle\\
            &\qquad=-\frac{u}{u+\varepsilon}\mathcal R
                +\frac{|\nabla u|^2}{(u+\varepsilon)^2}
                -\frac{2|\nabla u|^2}{u(u+\varepsilon)}\\
            &\qquad=-\frac{u(u+2\varepsilon)}{(u+\varepsilon)^2}
                    |\nabla\log u|^2
                -\frac{u}{u+\varepsilon}\mathcal R.
        \end{aligned}
    \]
    Noting that
    \[
        0\le\frac{u}{u+\varepsilon}\le1
        \quad\text{and}\quad
        0\le\frac{u(u+2\varepsilon)}{(u+\varepsilon)^2}
            =1-\frac{\varepsilon^2}{(u+\varepsilon)^2}\le1,
    \]
    we have $|q_\varepsilon|\le|\nabla\log u|^2+|\mathcal R|$.
    By Propositions~\ref{prop fisher finite} and \ref{prop scalar comparison},
    \[
        [|\nabla\log u|^2]_\tau=I(\tau)\le F(T)+K_*,
        \qquad
        [|\mathcal R|]_\tau\le[\mathcal R]_T+2K_*.
    \]
    In particular, $\int_0^T\bigl(I(\tau)+[|\mathcal R|]_\tau\bigr)\,d\tau
            <+\infty$.
    Lemma~\ref{lem bounded chain pairing} therefore applies to
            $v_\varepsilon$ and hence, for $0\le s<t\le T$, we have
    \begin{equation}
        \label{eq truncated entropy dissipation}
        [v_\varepsilon]_s=[v_\varepsilon]_t
            +\int_s^t\left[
                \frac{u(u+2\varepsilon)}{(u+\varepsilon)^2}
                    |\nabla\log u|^2
                +\frac{u}{u+\varepsilon}\mathcal R
                \right]_\tau\,d\tau.
    \end{equation}

    We now set $t=T$ and let $\varepsilon\to0^+$.
    By monotone convergence,
    \[
        [v_\varepsilon]_T
        =\int_M\phi^2\log\frac{U+\varepsilon}{\phi^2+\varepsilon}
                \,d\mu_{g(T)}
        \to
            \int_M\phi^2\log\frac{U}{\phi^2}\,d\mu_{g(T)}<+\infty.
    \]
    By dominated convergence,
    \[
        \int_s^T\left[
            \frac{u(u+2\varepsilon)}{(u+\varepsilon)^2}
                |\nabla\log u|^2
            +\frac{u}{u+\varepsilon}\mathcal R
            \right]_\tau\,d\tau
        \to\int_s^TF(\tau)\,d\tau.
    \]
    Consequently, by monotone convergence,
            \eqref{eq truncated entropy dissipation} becomes
    \[
        \int_Mu(s)\log\frac{U}{u(s)}\,d\mu_{g(s)}
        =\int_M\phi^2\log\frac{U}{\phi^2}\,d\mu_{g(T)}
            +\int_s^TF(\tau)\,d\tau<+\infty.
    \]
    This holds for every $s\in[0,T]$.
    Then
    \[
        \int_Mu(s)|\log u(s)|\,d\mu_{g(s)}
        \le\int_Mu(s)\log\frac{U}{u(s)}\,d\mu_{g(s)}
            +|\log U|<+\infty.
    \]
    Thus $H(s)$ is well defined, and the above identity becomes
    \[
        H(s)=H(T)+\int_s^TF(\tau)\,d\tau,
    \]
    and \eqref{eq entropy dissipation} follows.
    Since $F$ is bounded by Proposition~\ref{prop fisher finite},
            $H$ is absolutely continuous on $[0,T]$.
\end{proof}

\subsection{\texorpdfstring{$\mathcal W$}{W}-entropy and entropy monotonicity}

Following Wang \cite{Wang18CambJ,Wang20arXiv}, for an open set
        $\Omega\subset M$, let
\[
    \mathscr S(\Omega,g):=\left\{\varphi\in C_c^\infty(\Omega)\ \middle|\
            \varphi\ge0,\ \int_\Omega\varphi^2\,d\mu_g=1\right\}.
\]
For $\varphi\in\mathscr S(\Omega,g)$ and $\tau>0$, set
\begin{equation}
    \label{eq:W-definition}
    \begin{split}
        \mathcal W(\Omega,g,\varphi,\tau)
                &:=\int_\Omega\left(4\tau|\nabla\varphi|_g^2
                +\tau\mathcal R_g\varphi^2
                -\varphi^2\log\varphi^2\right)d\mu_g\\
        &\qquad-n-\frac n2\log(4\pi\tau),
    \end{split}
\end{equation}
and define
\[
    \begin{aligned}
        \mu(\Omega,g,\tau)&:=\inf\limits_{\varphi\in\mathscr S(\Omega,g)}
                \mathcal W(\Omega,g,\varphi,\tau),\\
        \nu(\Omega,g,\tau)&:=\inf\limits_{0<s\le\tau}\mu(\Omega,g,s).
    \end{aligned}
\]
We omit $\Omega$ when $\Omega=M$.
Note that Wang defines $\mu$ using normalized nonnegative
        $W^{1,2}_0(\Omega)$ functions, whereas we use
        $C_c^\infty(\Omega)$ functions.
For $\Omega\subset\subset M$, the two definitions of $\mu$ and $\nu$
        agree, see also the approximation argument in the proof of
        Lemma~\ref{lem:entropy-rigidity}.
For a probability density $u$ satisfying
\[
    \sqrt u\in W^{1,2}(M,g)
    \quad\text{and}\quad
    \int_M\bigl(|\mathcal R_g|u+u|\log u|\bigr)\,d\mu_g<+\infty,
\]
we use the same integral formula to define
        $\mathcal W(g,\sqrt u,\tau)$, without requiring compact support.
For a smooth positive density $u$, let $f:=-\log u-\frac n2\log(4\pi\tau)$.
Then $u=(4\pi\tau)^{-n/2}e^{-f}$, and the above definition agrees
        with Perelman's original functional \cite[Section~3]{Perelman02arXiv}.

If $\mathcal R_{g(t)}\ge-K_*$ on $M\times[0,T]$ for some $K_*\ge0$,
        then, by Propositions~\ref{prop scalar comparison},
        \ref{prop fisher finite}, and \ref{prop entropy dissipation},
        \eqref{eq:W-definition} is finite for the density
        $u(t)$ defined by \eqref{eq conjugate definition}, and we have
\begin{equation}
    \label{eq entropy functional}
    \mathcal W(g(t),\sqrt{u(t)},\tau)=\tau F(t)+H(t)-\frac n2\log(4\pi\tau)-n.
\end{equation}
At the terminal time, the energy is interpreted as in
        \eqref{eq terminal fisher}.

For a conjugate heat kernel $u=K(x_0,T;\cdot,t)$, the quantity
        $H(t)-\frac n2\log(4\pi(T-t))-\frac n2$ is the pointed
        Nash entropy based at $(x_0,T)$, see \cite[Section~1]{HN14CPAM} and
        \cite[Definition~5.1]{Bamler20arXiv}.

On a closed Ricci flow, Perelman's monotonicity formula implies
        that $t\mapsto\mathcal W(g(t),\sqrt{u(t)},\tau+T-t)$ is
        nondecreasing for every $\tau>0$ and every normalized
        positive solution $u$ of the conjugate heat equation, see
        \cite[Section~3, equation~(3.4)]{Perelman02arXiv} and
        \cite[equation~(17.12)]{CGIKLN10Book}.
For noncompact results under global geometric bounds, see
        \cite[Theorem~7.1]{CTY11CJM}, \cite{Zhang12JFA}, and
        \cite[Section~3.3]{LNN23GT}.
The following theorem establishes the corresponding entropy monotonicity
        for complete Ricci flows under an initial scalar curvature lower bound.
Its proof uses the $F$-energy monotonicity and $H$-entropy identity established in this section.

\begin{theorem}
    \label{thm entropy monotonicity}
    Let $(M,g(t))$, $t\in[0,T]$, be a complete Ricci flow
            satisfying $\mathcal R_{g(0)}\ge-K_*$ for some $K_*\ge0$.
    Let $\phi\in\mathscr S(M,g(T))$, and let $u$ be the minimal
            nonnegative solution of the conjugate heat equation
            with $u(T)=\phi^2$, given by \eqref{eq conjugate definition}.
    Then, for every $\tau>0$,
    \begin{equation}
        \label{eq entropy monotonicity}
        \mathcal W(g(T),\phi,\tau)
        \ge\mathcal W(g(0),\sqrt{u(0)},\tau+T).
    \end{equation}
\end{theorem}

\begin{proof}
    By Chen's estimate \eqref{eq:chen-scalar}, we have
            $\mathcal R\ge-K_*$ on $M\times[0,T]$.
    Fix $\tau>0$ and set
    \begin{align*}
        \mathcal W(t):={}&\mathcal W(g(t),\sqrt{u(t)},\tau+T-t)\\
        ={}&(\tau+T-t)F(t)+H(t)-\frac n2\log\bigl(4\pi(\tau+T-t)\bigr)-n,
    \end{align*}
    where the second equality follows from \eqref{eq entropy functional}.
    By Propositions~\ref{prop fisher finite},
            \ref{prop fisher differential}, and
            \ref{prop entropy dissipation}, the function $F$ is bounded
            and nondecreasing on $[0,T]$, while $H$ is absolutely
            continuous and satisfies $dH=-F\,dt$.
    Hence, as measures on $(0,T)$,
    \begin{align*}
        d\mathcal W&=(\tau+T-t)\,dF+\left(-2F+\frac{n}{2(\tau+T-t)}\right)dt\\
        &\ge\left(\frac{2(\tau+T-t)}nF^2-2F+\frac{n}{2(\tau+T-t)}\right)dt\\
        &=\frac{2(\tau+T-t)}n\left(F-\frac{n}{2(\tau+T-t)}\right)^2dt\ge0,
    \end{align*}
    where the second line uses \eqref{eq fisher differential}.
    The one-sided limits of $\mathcal W$ exist because $F$ is bounded
            and monotone and $H$ is continuous.
    Since $H$ is continuous and $\tau+T-t>0$, we conclude that
    \[
        \mathcal W(0)\le\mathcal W(0^+)\le\mathcal W(T^-)\le\mathcal W(T),
    \]
    which proves \eqref{eq entropy monotonicity}.
\end{proof}

\begin{corollary}
    \label{cor initial entropy propagation}
    Let $(M,g(t))$, $t\in[0,T]$, be a complete Ricci flow
            satisfying $\mathcal R_{g(0)}\ge-K_*$ for some $K_*\ge0$.
    Then, for every $t\in[0,T]$ and $\tau>0$,
    \begin{equation}
        \label{eq:transport}
        \mu(g(t),\tau)\ge\mu(g(0),t+\tau).
    \end{equation}
\end{corollary}

\begin{proof}
    The case $t=0$ is immediate.
    Fix $t\in(0,T]$, $\tau>0$, and $\varphi_t\in\mathscr S(M,g(t))$.
    By Chen's estimate \eqref{eq:chen-scalar},
            $\mathcal R_{g(s)}\ge-K_*$ for all $s\in[0,t]$.
    Let $u$ be given by \eqref{eq conjugate definition}, with $T=t$
            and terminal value $u(t)=\varphi_t^2$.
    By Theorem~\ref{thm entropy monotonicity},
    \[
        \mathcal W(g(t),\varphi_t,\tau)
        \ge\mathcal W(g(0),\sqrt{u(0)},t+\tau).
    \]
    Write $\varphi_0:=\sqrt{u(0)}$.
    Since $\varphi_t$ is arbitrary, by the definition of
            $\mu$, it suffices to show that
    \[
        \mathcal W(g(0),\varphi_0,t+\tau)\ge\mu(g(0),t+\tau).
    \]
    Since $\varphi_0$ need not have compact support, we must justify the
            inequality.

    In the remainder of the proof, all spatial derivatives, norms,
            and volume measures are taken with respect to $g(0)$.
    By \eqref{eq conjugate equation} and
            Propositions~\ref{prop scalar comparison},
            \ref{prop fisher finite}, and \ref{prop entropy dissipation},
            we have $\int_M\varphi_0^2\,d\mu=1$ and
    \begin{equation}
        \label{eq:W-integrability}
        \int_M\left(|\nabla\varphi_0|^2+|\mathcal R_{g(0)}|\varphi_0^2
                +\varphi_0^2|\log \varphi_0^2|\right)d\mu<+\infty.
    \end{equation}

    Choose $\zeta_j\in C_c^\infty(M)$ such that
    \[
        0\le\zeta_j\le1,\quad \zeta_j\to1\text{ pointwise},
        \quad\text{and}\quad
        \|\nabla\zeta_j\|_{L^\infty}\to0.
    \]
    For all sufficiently large $j$, set
    \[
        q_j:=\int_M\zeta_j^2\varphi_0^2\,d\mu\to1,
        \qquad
        \phi_j:=q_j^{-1/2}\zeta_j\varphi_0.
    \]
    Since $t>0$, the function $\varphi_0$ is positive and smooth.
    Hence $\phi_j\in\mathscr S(M,g(0))$.
    By \eqref{eq:W-definition},
    \[
        \begin{aligned}
            \mathcal W(g(0),\phi_j,t+\tau)
            ={}&(t+\tau)\int_M\left(4|\nabla\phi_j|^2
                    +\mathcal R_{g(0)}\phi_j^2\right)d\mu\\
            &-\int_M\phi_j^2\log\phi_j^2\,d\mu
                    -n-\frac n2\log\bigl(4\pi(t+\tau)\bigr).
        \end{aligned}
    \]
    By \eqref{eq:W-integrability} and the properties of $\zeta_j$, we have
    \[
        \nabla(\zeta_j\varphi_0)-\nabla\varphi_0
        =(\zeta_j-1)\nabla\varphi_0+\varphi_0\nabla\zeta_j\to0
        \quad\text{in }L^2(M,d\mu).
    \]
    Since $q_j\to1$, the gradient term converges.
    The scalar curvature term converges by dominated convergence.
    For the entropy term, we have
    \begin{align*}
        \int_M\phi_j^2\log\phi_j^2\,d\mu
        &=q_j^{-1}\int_M\zeta_j^2\varphi_0^2\log\varphi_0^2\,d\mu\\
        &\qquad+q_j^{-1}\int_M\varphi_0^2\zeta_j^2\log\zeta_j^2\,d\mu
            -\log q_j.
    \end{align*}
    Since $\varphi_0^2|\log\varphi_0^2|\in L^1(M,d\mu)$ by
            \eqref{eq:W-integrability} and
            $|\varphi_0^2\zeta_j^2\log\zeta_j^2|\le\varphi_0^2/e$,
            dominated convergence implies that the first integral tends to
            $\int_M\varphi_0^2\log\varphi_0^2\,d\mu$ and the second to zero.
    Together with $q_j\to1$, we conclude that
    \[
        \begin{aligned}
            \mathcal W(g(0),\varphi_0,t+\tau)
            &=\lim\limits_{j\to\infty}
                \mathcal W(g(0),\phi_j,t+\tau)\\
            &\ge\mu(g(0),t+\tau),
        \end{aligned}
    \]
    which completes the proof.
\end{proof}

\section{Uniform pseudolocality}
\label{sec:curvature}

\subsection{Small-scale entropy of the initial metric}

By the sharp Euclidean logarithmic Sobolev inequality
        \cite[Theorem~22.15]{CGIKLN10Book}, we have
        $\mathcal W(g_{\mathrm E},\varphi,\tau)\ge0$ for every
        $\varphi\in\mathscr S(\mathbb R^n,g_{\mathrm E})$ and $\tau>0$.
By taking normalized compact cutoffs of the Gaussian
        $(4\pi\tau)^{-n/4}e^{-|x|^2/(8\tau)}$ as test functions, we obtain
\[
    \mu(g_{\mathrm E},\tau)=0\quad\text{for every}\quad\tau>0.
\]

Let $(M^n,g)$ be a complete Riemannian manifold and
        $\Omega\subset\subset M$ be a smooth domain.
Define
\[
    \mathbf I(\Omega,g):=\inf\limits_{D\subset\subset\Omega}
            \frac{\Area_g(\partial D)}
            {\Vol_g(D)^{(n-1)/n}}
    \quad\text{and}\quad
    \mathbf I_n:=n\omega_n^{1/n},
\]
where $\omega_n$ is the volume of the Euclidean unit ball.
If $\mathbf I(\Omega,g)>0$ and $\mathcal R_g\ge-\Lambda$ on $\Omega$
        for some $\Lambda\ge0$, then by Wang's isoperimetric comparison
        \cite[Lemma~3.5]{Wang18CambJ} and the Euclidean logarithmic
        Sobolev inequality, we have
\begin{equation}
    \label{eq Wang local comparison}
    \mu(\Omega,g,\tau)
    \ge n\log\frac{\mathbf I(\Omega,g)}{\mathbf I_n}-\Lambda\tau
    \quad\text{for every}\quad\tau>0.
\end{equation}
See also \cite[Lemma~2.9]{Wang20arXiv}.

For a nonnegative $\varphi\in C_c^\infty(M)$, put
        $q:=\int_M\varphi^2\,d\mu_g$ and define
\begin{align*}
    \mathcal W_0(g,\varphi,\tau)
    &:=\int_M\left(4\tau|\nabla\varphi|_g^2
            +\tau\mathcal R_g\varphi^2
            -\varphi^2\log\varphi^2\right)d\mu_g\\
    &\qquad+q\log q-\left(n+\frac n2\log(4\pi\tau)\right)q.
\end{align*}
For $q>0$, this equals $q\mathcal W(g,q^{-1/2}\varphi,\tau)$;
        for $q=0$, it equals zero.
It agrees with $\mathcal W$ when $q=1$.
In this form, local entropy bounds apply directly to localized test
        functions of arbitrary, including zero, $L^2$ norm.

\begin{lemma}
    \label{lem uniform small entropy}
    Let $(M^n,g)$, $n\ge2$, be a complete Riemannian manifold
            satisfying $\Ric_g\ge-K$ and $\inj_g\ge i_0$
            for some $K\ge0$ and $i_0>0$.
    For every $\varepsilon\in(0,1)$, there is a constant
            $C(n,\varepsilon,Ki_0^2)>0$ such that, for every $\tau>0$,
    \begin{equation}
        \label{eq uniform entropy bound}
        \mu(g,\tau)
        \ge-n(n-1)\log(1+\varepsilon)-\frac{\tau}{i_0^2}\left(nKi_0^2+4C\right).
    \end{equation}
    In particular,
    \[
        \liminf\limits_{\delta\to0^+}\nu(g,\delta)\ge0.
    \]
\end{lemma}

\begin{proof}
    We first obtain uniform entropy bounds on coordinate domains and
            then transfer them to $M$ using cutoffs with
            $\sum\limits_j\eta_j^2=1$.
    Fix $\varepsilon\in(0,1)$.
    By the local $W^{1,2n}$ harmonic-radius estimate of Anderson-Cheeger
            \cite[Theorem~0.3]{AC92JDG}, there exists
            $\rho\in(0,i_0/4)$, with $\rho/i_0$ depending only on
            $n,\varepsilon,Ki_0^2$, such that every $B_g(p,4\rho)$
            admits harmonic coordinates satisfying
    \begin{equation}
        \label{eq initial coordinate comparison}
        (1+\varepsilon)^{-1}g_{\mathrm E}
        \le g\le(1+\varepsilon)g_{\mathrm E}.
    \end{equation}
    Here $g_{\mathrm E}$ denotes the Euclidean metric in these coordinates.

    Choose a maximal $\rho$-separated set $\{p_j\}$.
    Then the balls $B_g(p_j,\rho)$ cover $M$, while
            $B_g(p_j,\rho/2)$ are pairwise disjoint.
    Put $\Omega_j:=B_g(p_j,3\rho)$.
    Fix $x\in M$.
    If $x\in B_g(p_j,2\rho)$, then
    \[
        B_g(p_j,\rho/2)\subset B_g(x,5\rho/2)
        \subset B_g(p_j,9\rho/2).
    \]
    By Bishop-Gromov volume comparison and $\rho<i_0/4$, we have
    \[
        \VolB_g(p_j,\rho/2)
        \ge c\VolB_g(p_j,9\rho/2)
        \ge c\VolB_g(x,5\rho/2),
    \]
    for some $c(n,Ki_0^2)>0$.
    Since the balls $B_g(p_j,\rho/2)$ are disjoint,
            every point of $M$ belongs to at most $N:=\lfloor c^{-1}\rfloor$
            of the balls $B_g(p_j,2\rho)$.

    Choose smooth distance cutoffs $0\le\widetilde\eta_j\le1$ with
    \[
        \widetilde\eta_j=1\text{ on }B_g(p_j,\rho),
        \quad\supp\widetilde\eta_j\subset\subset B_g(p_j,2\rho)
        \quad\text{and}\quad
        |\nabla\widetilde\eta_j|\le c_1/\rho,
    \]
    where $c_1>0$ is a universal constant.
    These cutoffs are smooth because $2\rho<i_0$ and they are
            constant near their centers.
    The family $\{\widetilde\eta_j\}$ is locally finite.
    Set
    \[
        S:=\sum_j\widetilde\eta_j^2\ge1
        \quad\text{and}\quad
        \eta_j:=S^{-1/2}\widetilde\eta_j.
    \]
    Direct differentiation gives
    \[
        \sum_j|\nabla\eta_j|^2
        =\frac{\sum_j|\nabla\widetilde\eta_j|^2}{S}
                -\frac{|\nabla S|^2}{4S^2}
        \le\frac{Nc_1^2}{\rho^2}.
    \]
    Thus, $\eta_j\in C_c^\infty(\Omega_j)$ and
    \[
        \sum_j\eta_j^2=1
        \quad\text{and}\quad
        \sum_j|\nabla\eta_j|^2\le\frac{C}{i_0^2},
    \]
    for some $C(n,\varepsilon,Ki_0^2)>0$.

    By \eqref{eq initial coordinate comparison}, for every smooth
            $D\subset\subset\Omega_j$, we have
    \begin{align*}
        \Area_g(\partial D)
        &\ge(1+\varepsilon)^{-\frac{n-1}{2}}
                \Area_{\mathrm E}(\partial D),\\
        \Vol_g(D)
        &\le(1+\varepsilon)^\frac{n}{2}\Vol_{\mathrm E}(D),
    \end{align*}
    where $\Area_{\mathrm E}$ and $\Vol_{\mathrm E}$ are computed
            in the coordinate image.
    Hence
    \[
        \mathbf I(\Omega_j,g)
        \ge(1+\varepsilon)^{-(n-1)}\mathbf I_n.
    \]
    Since $\mathcal R_g\ge-nK$, by Wang's estimate
            \eqref{eq Wang local comparison}, we have
    \[
        \mu(\Omega_j,g,\tau)
        \ge-n(n-1)\log(1+\varepsilon)-nK\tau.
    \]
    For a nonzero nonnegative $\psi\in C_c^\infty(\Omega_j)$, set
            $q:=\int_M\psi^2\,d\mu_g>0$.
    Then $q^{-1/2}\psi\in\mathscr S(\Omega_j,g)$ and hence
    \[
        \mathcal W_0(g,\psi,\tau)
        =q\mathcal W(g,q^{-1/2}\psi,\tau)
        \ge q\mu(\Omega_j,g,\tau).
    \]
    Consequently,
    \begin{equation}
        \label{eq chart entropy comparison}
        \mathcal W_0(g,\psi,\tau)
        \ge\left[-n(n-1)\log(1+\varepsilon)
                -nK\tau\right]
                \int_M\psi^2\,d\mu_g
    \end{equation}
    for every nonnegative $\psi\in C_c^\infty(\Omega_j)$ and $\tau>0$.
    The case $\psi=0$ is immediate.

    Fix $\varphi\in\mathscr S(M,g)$ and set $\varphi_j:=\eta_j\varphi$ and
            $q_j:=\int_M\varphi_j^2\,d\mu_g$.
    Only finitely many $\varphi_j$ are nonzero, and $\sum\limits_jq_j=1$.
    We differentiate $\sum\limits_j\eta_j^2=1$ to obtain
            $\sum\limits_j\eta_j\nabla\eta_j=0$, and hence
    \[
        \sum_j\int_M|\nabla\varphi_j|^2\,d\mu_g
        =\int_M|\nabla\varphi|^2\,d\mu_g
                +\int_M\varphi^2\sum_j|\nabla\eta_j|^2\,d\mu_g.
    \]
    Moreover,
    \[
        \sum_j\varphi_j^2\log\varphi_j^2
        =\varphi^2\log\varphi^2
                +\varphi^2\sum_j\eta_j^2\log\eta_j^2.
    \]
    Since $\varphi^2d\mu_g$ is a probability measure and $z\log z$
            is convex, by Jensen's inequality, we obtain
    \[
        \sum_jq_j\log q_j
        \le\int_M\varphi^2\sum_j\eta_j^2\log\eta_j^2\,d\mu_g.
    \]
    Therefore,
    \begin{align*}
        \sum_j\mathcal W_0(g,\varphi_j,\tau)
        &=\mathcal W(g,\varphi,\tau)
                +4\tau\int_M\varphi^2\sum_j|\nabla\eta_j|^2\,d\mu_g\\
        &\qquad+\sum_jq_j\log q_j
                -\int_M\varphi^2\sum_j\eta_j^2\log\eta_j^2\,d\mu_g\\
        &\le\mathcal W(g,\varphi,\tau)+\frac{4C\tau}{i_0^2}.
    \end{align*}
    We sum \eqref{eq chart entropy comparison} and use $\sum\limits_jq_j=1$
            to obtain
    \[
        \mathcal W(g,\varphi,\tau)
        \ge-n(n-1)\log(1+\varepsilon)
                -nK\tau-\frac{4C\tau}{i_0^2}.
    \]
    Since $\varphi\in\mathscr S(M,g)$ is arbitrary,
            \eqref{eq uniform entropy bound} follows.

    The bound holds uniformly for $0<\tau\le\delta$.
    We take the infimum over these $\tau$, let $\delta\to0^+$,
            and then let $\varepsilon\to0^+$ to obtain
            $\liminf\limits_{\delta\to0^+}\nu(g,\delta)\ge0$.
    This completes the proof.
\end{proof}

By Lemma~\ref{lem uniform small entropy} and
        Corollary~\ref{cor initial entropy propagation}, we obtain a lower
        bound for $\mu(g(t),\tau)$, whose error tends to zero as $t+\tau\to0^+$,
        uniformly under \eqref{eq:initial-geometry}.
These entropy estimates will be applied to curvature estimates in
        Section~\ref{Curvature estimates from entropy}.

\subsection{Entropy rigidity}

For the reader's convenience, we present the following special case of
        Cheng's local entropy rigidity theorem
        \cite[Theorem~1.5(ii)]{Cheng24arXiv}, corresponding to $Q=0$.
No completeness assumption is required.

\begin{lemma}[Cheng]
    \label{lem:entropy-rigidity}
    Let $(M^n,g)$ be a Riemannian manifold.
    If $\mu(g,\tau)\ge0$ for every $\tau>0$, then $g$ is flat.
\end{lemma}

\begin{proof}
    For every smooth domain $\Omega\subset\subset M$, we have
    \[
        \mu(\Omega,g,\tau)\ge\mu(g,\tau)\ge0
        \quad\text{for every}\quad\tau>0.
    \]
    This lower bound also holds for Cheng's definition of local
            entropy, which uses normalized positive
            $W_0^{1,2}(\Omega)$ test functions.
    Indeed, under smooth approximation, the gradient and
            scalar curvature integrals converge, since $\Omega$ is relatively
            compact.
    For the entropy term, note that for any $\varepsilon>0$, there exists
            $C_\varepsilon>0$ such that
    \[
        \left|\frac{d}{ds}(s^2\log s^2)\right|
        \le C_\varepsilon(1+s^{1+\varepsilon}).
    \]
    For $\varepsilon>0$ sufficiently small, we have $L^{2+\varepsilon}$
            convergence by Sobolev embedding, and the convergence of the entropy
            integrals follows from the mean value theorem and H\"older's
            inequality.
    By applying Cheng's theorem with $Q=0$, we obtain $g$
            is flat on $\Omega$.
    Since $\Omega$ is arbitrary, $g$ is flat on $M$.
\end{proof}

\subsection{Pseudolocality from entropy}
\label{Curvature estimates from entropy}

We first state the curvature estimate under a uniform initial
        lower bound for the $\mu$-functional.
Related estimates under bounded curvature hypotheses appear in
        \cite[Theorem~1.2]{Wang20arXiv} and
        \cite[Theorem~3.2]{LNN23GT}, here no curvature bound is assumed
        along the flow.

\begin{theorem}
    \label{thm:curvature}
    Given $n\ge2$, $\alpha>0$, $K_*\ge0$, $s_*>0$ and a function
            $\omega:(0,s_*]\to[0,+\infty)$ satisfying
            $\lim\limits_{s\to0^+}\omega(s)=0$, there is a constant
            $\delta(n,\alpha,K_*,s_*,\omega)>0$ with the following property.
    Let $(M^n,g_0)$ be a complete Riemannian manifold satisfying
    \begin{equation}
        \label{eq:entropy-initial}
        \mathcal R_{g_0}\ge-K_*,\qquad
        \mu(g_0,s)\ge-\omega(s)
        \quad\text{for every}\quad s\in(0,s_*].
    \end{equation}
    Then any complete Ricci flow $g(t)$ on $[0,T]$ with $g(0)=g_0$ satisfies
    \[
        |\Rm|_{g(t)}\le\frac{\alpha}{t}
        \quad\text{on}\quad M\times(0,\min\{\delta,T\}].
    \]
\end{theorem}

\begin{proof}
    Suppose the theorem is false.
    Then there exist times $t_i\to0^+$, complete flows
            $(M_i,\tilde g_i(t))$ on $[0,t_i]$ whose initial metrics satisfy
            \eqref{eq:entropy-initial} with the same $K_*$, $s_*$ and
            $\omega$, and points $p_i\in M_i$ such that
    \[
        Q_i:=|\Rm|_{\tilde g_i(t_i)}(p_i)>\frac{\alpha}{t_i}\to+\infty.
    \]

    Write $d_i(x,t):=d_{\tilde g_i(t)}(p_i,x)$. We use a variant of the weighted point-picking argument in \cite{Wang20arXiv}. Consider
    \[
        q_i(x,t):=t|\Rm|_{\tilde g_i(t)}(x)[1-d_i(x,t)]_+^2
        \quad\text{on}\quad M_i\times[0,t_i].
    \]
   For similar spatial weights, see also \cite[Lemma~4.3]{Ye07Kappa} and
            \cite[Theorem~1.6]{Simon17JGA}.
    For each fixed $i$, the moving unit balls about $p_i$ lie in a
            common compact subset of $M_i$.
    Therefore $q_i$ attains its positive maximum at a point
            $(\bar p_i,\bar t_i)$, where $d_i(\bar p_i,\bar t_i)<1$ and
            $0<\bar t_i\le t_i$.
    Set
    \[
        \bar Q_i:=|\Rm|_{\tilde g_i(\bar t_i)}(\bar p_i)\quad\text{and}\quad
        b_i:=1-d_i(\bar p_i,\bar t_i)\in(0,1].
    \]
    Then we have
    \[
        \bar t_i\bar Q_i b_i^2\ge t_iQ_i.
    \]
    Since $\bar t_i\le t_i$ and $b_i\le1$, it follows that
    \begin{equation}
        \label{eq:picked-scale}
        \bar t_i\bar Q_i\ge\alpha\quad\text{and}\quad
        b_i\bar Q_i^\frac{1}{2}\ge Q_i^\frac{1}{2}\to+\infty.
    \end{equation}
    The maximality at $(\bar p_i,\bar t_i)$ also implies
    \begin{equation}
        \label{eq:moving-region}
        |\Rm|_{\tilde g_i(s)}(y)\le32\bar Q_i
    \end{equation}
    whenever $\frac12\bar t_i\le s\le\bar t_i$ and $d_i(y,s)\le1-\frac14b_i$.

    Set $A_i:=\frac14b_i\bar Q_i^\frac{1}{2}\to+\infty$.
    We claim that the bound \eqref{eq:moving-region} holds on
    \[
        B_{\tilde g_i(\bar t_i)}\left(\bar p_i,A_i\bar Q_i^{-\frac{1}{2}}\right)
        \times\left[\bar t_i-\frac{\alpha}{4}\bar Q_i^{-1},\bar t_i\right],
    \]
    for all sufficiently large $i$.
    Indeed, by \eqref{eq:picked-scale}, we have
            $\bar t_i-\frac{\alpha}{4}\bar Q_i^{-1}\ge\frac34\bar t_i$.
    Fix $s\in[\bar t_i-\frac{\alpha}{4}\bar Q_i^{-1},\bar t_i)$.
    Then for all $t\in(s,\bar t_i]$, it holds that
    \[
        |\Rm|_{\tilde g_i(t)}\le32\bar Q_i\le\frac{32\bar Q_i(\bar t_i-s)}{t-s}
        \quad\text{on}\quad
        B_{\tilde g_i(t)}\left(p_i,1-\frac{b_i}{4}\right).
    \]
    By the shrinking balls lemma \cite[Corollary~3.3]{ST22JDG}, there exists
            a constant $C(n)>0$ such that
    \[
        B_{\tilde g_i(\bar t_i)}
            \left(p_i,1-\frac{b_i}{4}
                -C\bar Q_i^\frac{1}{2}(\bar t_i-s)\right)
        \subset
        B_{\tilde g_i(s)}\left(p_i,1-\frac{b_i}{4}\right).
    \]
    Since $\bar t_i-s\le\frac{\alpha}{4}\bar Q_i^{-1}$ and
            $b_i\sqrt{\bar Q_i}\to+\infty$, for all sufficiently large $i$,
            we have
    \[
        1-b_i+A_i\bar Q_i^{-\frac{1}{2}}=1-\frac{3b_i}{4}
                <1-\frac{b_i}{4}-C\bar Q_i^\frac{1}{2}(\bar t_i-s).
    \]
    By the triangle inequality, we obtain
    \[
        B_{\tilde g_i(\bar t_i)}
            (\bar p_i,A_i\bar Q_i^{-1/2})
        \subset
        B_{\tilde g_i(s)}\left(p_i,1-\frac{b_i}{4}\right).
    \]
    Since $s$ is arbitrary in that interval and the case $s=\bar t_i$ is
            immediate, the claim follows.

    We then consider the rescaled flows
    \[
        g_i(t):=\bar Q_i\tilde g_i\left(\frac{t}{\bar Q_i}+\bar t_i\right),
        \qquad
        -\frac{\alpha}{4}\le t\le0.
    \]
    We have
    \begin{equation}
        \label{eq:rescaled-curvature}
        |\Rm|_{g_i(0)}(\bar p_i)=1\quad\text{and}\quad
        |\Rm|_{g_i(t)}\le32
    \end{equation}
    on $B_{g_i(0)}(\bar p_i,A_i)\times[-\frac{\alpha}{4},0]$
            for all sufficiently large $i$.

    For every fixed $\tau>0$, we have $\bar t_i+\tau/\bar Q_i\le s_*$ for all
            sufficiently large $i$.
    By Corollary~\ref{cor initial entropy propagation}
            and the initial entropy bound \eqref{eq:entropy-initial}, we obtain
    \[
        \mu(g_i(0),\tau)
        =\mu\left(\tilde g_i(\bar t_i),\frac{\tau}{\bar Q_i}\right)
        \ge\mu\left(\tilde g_i(0),\bar t_i+\frac{\tau}{\bar Q_i}\right)
        \ge-\omega\left(\bar t_i+\frac{\tau}{\bar Q_i}\right).
    \]
    It follows that
    \[
        \liminf\limits_{i\to\infty}\mu(g_i(0),\tau)\ge0.
    \]
    The estimate is uniform for all $\tau\in(0,1]$, so
    \[
        \nu(g_i(0),1)\ge-1
    \]
    for all sufficiently large $i$.

    By the curvature bound \eqref{eq:rescaled-curvature},
            domain monotonicity \cite[Proposition 2.1]{Wang18CambJ}
            and the volume estimate \cite[Theorem~3.3]{Wang18CambJ}, we obtain
            a uniform lower bound on the volume of $B_{g_i(0)}(\bar p_i,1)$.
    The injectivity radius lower bound is then obtained from
            Cheeger-Gromov-Taylor \cite{CGT82JDG}.
    Since $A_i\to+\infty$, Hamilton's compactness theorem \cite{Hamilton95AJM},
            see also \cite[Theorem~1.6]{Topping14Crelle}, implies that,
            after passing to a subsequence,
            $(M_i,g_i(t),\bar p_i)$ converges smoothly to a complete
            solution $(M_\infty,g_\infty(t),p_\infty)$ on
            $(-\frac{\alpha}{4},0]$ with bounded curvature
            and $|\Rm|_{g_\infty(0)}(p_\infty)=1$.

    By domain monotonicity and \cite[Corollary~2.5]{Wang18CambJ},
            $\mu(\Omega,g_\infty(0),\tau)\ge0$ for every smooth domain
            $\Omega\subset\subset M_\infty$ and every $\tau>0$.
    Since such domains exhaust $M_\infty$, we have
            $\mu(g_\infty(0),\tau)\ge0$ for every $\tau>0$.
    Lemma~\ref{lem:entropy-rigidity} implies that $g_\infty(0)$ is flat,
            contradicting $|\Rm|_{g_\infty(0)}(p_\infty)=1$.
    This completes the proof.
\end{proof}

\begin{proof}[Proof of Theorem~\ref{thm:main-curvature}]
    For every $\varepsilon\in(0,1)$, let $C_\varepsilon>0$ be the constant in
            Lemma~\ref{lem uniform small entropy} corresponding to $\varepsilon$
            and the given $n,K,i_0$.
    The function
    \[
        \omega_{n,K,i_0}(s):=\inf\limits_{0<\varepsilon<1}
                \left\{n(n-1)\log(1+\varepsilon)+\frac{s}{i_0^2}
                \bigl(nKi_0^2+4C_\varepsilon\bigr)\right\}
    \]
    is nonnegative and tends to zero as $s\to0^+$.
    By Lemma~\ref{lem uniform small entropy},
            $\mu(g_0,s)\ge-\omega_{n,K,i_0}(s)$ for every $s>0$.
    Since $\mathcal R_{g_0}\ge-nK$, by Theorem~\ref{thm:curvature}
            with $K_*=nK$ and $s_*=i_0^2$, we obtain
            \eqref{eq:uniform-curvature} with
    \[
        \hat T(n,\alpha,K,i_0)
        :=\delta(n,\alpha,nK,i_0^2,\omega_{n,K,i_0})>0.
    \]
    This completes the proof.
\end{proof}

\section{Strong uniqueness}
\label{sec:consequences}

In this section, we prove the strong uniqueness results stated in the
        introduction by applying Theorem~\ref{thm:main-curvature}.

\begin{proof}[Proof of Corollary~\ref{cor:uniqueness}]
    By Theorem~\ref{thm:main-curvature} with $\alpha=1$ applied to both
            flows, we have
    \[
        |\Rm|_{g(t)}+|\Rm|_{\widetilde g(t)}\le\frac2t
    \]
    on $0<t\le\min\{\hat{T}(n,1,K,i_0),T\}$.
    If $M$ is noncompact, by Lee's uniqueness theorem
            \cite[Theorem~1.1]{Lee25arXiv}, we obtain
            $\widetilde g(t)\equiv g(t)$ on
            $0<t\le\min\{\hat{T}(n,1,K,i_0),T\}$.
    If $M$ is compact, the conclusion follows from classical uniqueness
            for Ricci flow on compact manifolds \cite{Hamilton82JDG}.
\end{proof}

\begin{proof}[Proof of Corollary~\ref{cor:ricci-flat}]
    Choose $i_0>0$ with $\inj_{g_0}\ge i_0$ and compare the given flow $g(t)$
            with the stationary complete flow $g_0$.
    By Corollary~\ref{cor:uniqueness} with $K=0$, we have $g(t)\equiv g_0$ on
            $[0, \min\{\hat{T}(n,1,0,i_0),T\}]$.
    Since $\hat{T}$ is uniform, $g(t)\equiv g_0$ on $[0,T]$.
\end{proof}

\begin{proof}[Proof of Corollary~\ref{cor flat uniqueness}]
    By Corollary~\ref{cor:ricci-flat}, if $g(0)=g_{\mathrm E}$ on
            $\mathbb{R}^n$, then $g(t)\equiv g_{\mathrm E}$ throughout its
            existence time interval.
    For a general complete flat initial metric, we pull back the flow to the
            universal cover by a fixed covering map $\pi$.
    The lifted initial metric $\pi^*g_0$ is complete and flat, so the universal
            cover with this metric is isometric to Euclidean space.
    The lifted flow $\pi^*g(t)$ is also complete and hence is stationary by
            the above argument.
    Since $\pi$ is a surjective local diffeomorphism, $\pi^*g(t)=\pi^*g_0$
            implies $g(t)=g_0$ throughout its existence time interval.
    This completes the proof.
\end{proof}

\bibliographystyle{amsplain}
\bibliography{refs2}

\end{document}